\documentclass[reqno,11pt]{amsart}

\usepackage{amsmath, amssymb ,amsthm, amsfonts, amsgen,color}
\usepackage{bbm}
\usepackage{bm}
\usepackage{hyperref}
\usepackage{mathtools}
\usepackage[dvipsnames]{xcolor}
\usepackage{tikz-cd}
\usepackage{tikz}
\usetikzlibrary{matrix,arrows,decorations.pathmorphing,decorations.markings,shapes,calc}
\usetikzlibrary{patterns} 
\usepackage{pgfplots}
\usepackage{caption}
\usepackage[noadjust]{cite}

\numberwithin{equation}{section}

\newcommand{\R}{{\mathbb{R}}}
\newcommand{\Zb}{{\mathbb{Z}}}

\newcommand{\C}{{\mathcal{C}}}
\newcommand{\N}{{\mathbb{N}}}

\newcommand{\beq}{\begin{equation}}
\newcommand{\eeq}{\end{equation}}

\newcommand{\Scal}{{\mathcal{S}}}

\newcommand{\Hcal}{{\mathcal{H}}}
\newcommand{\dd}{{\,\rm d}}

\newcommand{\Lcal}{{\mathcal{L}}}
\newcommand{\eps}{{\varepsilon}}

\newcommand{\one}{\mathbbm{1}}
\newcommand{\skw}{\mathrm{skew}}
\newcommand{\sym}{\mathrm{sym}}
\newcommand{\supp}{\mathrm{supp}\,}
\newcommand{\qand}{\quad\text{and}\quad}

\newcommand{\Tr}{\mathrm{Tr}\,}
\newcommand{\Id}{\mathrm{Id}\,}

\newcommand{\mres}{\mathbin{\vrule height 1.6ex depth 0pt width
0.13ex\vrule height 0.13ex depth 0pt width 1.3ex}}

\def\dist{\text{dist}}

\def\Xint#1{\mathchoice
{\XXint\displaystyle\textstyle{#1}}%
{\XXint\textstyle\scriptstyle{#1}}%
{\XXint\scriptstyle\scriptscriptstyle{#1}}%
{\XXint\scriptscriptstyle\scriptscriptstyle{#1}}%
\!\int}
\def\XXint#1#2#3{{\setbox 0=\hbox{$#1{#2#3}{\int}$}
\vcenter{\hbox{$#2#3$}}\kern-.5\wd0}}
\def\dashint{\Xint-}

\makeatletter
\def\rightharpoonupfill@{\arrowfill@\relbar\relbar\rightharpoonup}

\renewcommand{\xrightharpoonup}[2][]{\ext@arrow
0359\rightharpoonupfill@{#1}{#2}} \makeatother

\newtheoremstyle{thmlemcorr}{10pt}{10pt}{\itshape}{}{\bfseries}{.}{10pt}{{\thmname{#1}\thmnumber{ #2}\thmnote{ (#3)}}}
\newtheoremstyle{remexample}{10pt}{10pt}{}{}{\bfseries}{.}{10pt}{{\thmname{#1}\thmnumber{ #2}\thmnote{ (#3)}}}
\newtheoremstyle{defi}{10pt}{10pt}{\itshape}{}{\bfseries}{.}{10pt}{{\thmname{#1}\thmnumber{ #2}\thmnote{ (#3)}}}

\theoremstyle{thmlemcorr}
\newtheorem{theorem}{Theorem}[section]
\newtheorem{lemma}[theorem]{Lemma}
\newtheorem{proposition}[theorem]{Proposition}

\theoremstyle{defi}
\newtheorem{definition}[theorem]{Definition}

\theoremstyle{remexample}
\newtheorem{example}[theorem]{Example}
\newtheorem{remarkx}[theorem]{Remark}
\newenvironment{remark}
  {\pushQED{\qed}\remarkx}
  {\popQED\endremarkx}

\newcommand{\BV}{{BV}}
\newcommand{\BD}{{BD}}

\newcommand{\SBV}{{\rm SBV}_0(Q_\eta;\R^n)}
\newcommand{\Av}{{\rm VA}_0(Q_\eta;\R^n)}
\newcommand{\PC}{{\rm VG}_0(Q_\eta;\R^n)}
\newcommand{\Ca}{{\rm PC}_0(Q_\eta;\R^n)}
\newcommand{\PR}{{\rm PR}_0(Q_\eta;\R^n)}
\newcommand{\J}{{\rm SJ}_0(Q_\eta;\R^n)}

\newcommand{\SBVs}{{\rm SBV}_0}
\newcommand{\Avs}{{\rm VA}_0}
\newcommand{\PCs}{{\rm VG}_0}
\newcommand{\Cas}{{\rm PC}_0}
\newcommand{\Js}{{\rm SJ}_0}

\title[Characterizing ${\rm BV}$- and ${\rm BD}$-ellipticity]{Characterizing $\rm\bm{BV}$- and $\rm\bm{BD}$-ellipticity for a class of positively $\bm{1}$-homogeneous surface energy densities}

\author{Dominik Engl}
\address{Mathematisch-Geographische Fakult\"at, Katholische Universit\"at Eichst\"att-Ingolstadt, Ostenstra{\ss}e 28, 85071 Eichst\"att, Germany}
\email{dominik.engl@ku.de}

\author{Carolin Kreisbeck}
\address{Mathematisch-Geographische Fakult\"at, Katholische Universit\"at Eichst\"att-Ingolstadt, Ostenstra{\ss}e 28, 85071 Eichst\"att, Germany}
\email{carolin.kreisbeck@ku.de}

\author{Marco Morandotti}
\address{Dipartimento di Scienze Matematiche ``G. L. Lagrange'', Politecnico di Torino, Corso Duca degli Abruzzi 24, 10129 Torino, Italy}
\email{marco.morandotti@polito.it}

\begin{document}

\maketitle
\thispagestyle{empty}

 \begin{abstract}
 
\begin{itemize}
	Lower semicontinuity of surface energies in integral form is known to be equivalent to $\rm BV$-ellipticity of the surface density. 
	In this paper, we prove that $\rm BV$-ellipticity coincides with the simpler notion of biconvexity for a class of densities that depend only on the jump height and jump normal, and are positively $1$-homogeneous in the first argument.
	The second main result is the analogous statement in the setting of bounded deformations, where we show that $\rm BD$-ellipticity reduces to symmetric biconvexity. 
	Our techniques are primarily inspired by constructions from the analysis of structured deformations and the general theory of free discontinuity problems.
\end{itemize}
 \vspace{8pt}
\vspace{8pt}

 \noindent\textsc{MSC (2020):}  49J45, secondary: 26B25, 9Q20, 70G75
 
 \noindent\textsc{Keywords:} interfacial energy, lower semicontinuity, $\rm BV$- and $\rm BD$-ellipticity, biconvexity
 
 \vspace{8pt}
 
 \noindent\textsc{Date:} \today.
 \end{abstract}

\section{Introduction}
In many applied problems, especially, but not restricted to, those in continuum mechanics, equilibrium configurations are obtained by minimizing interfacial energies. 
One typically studies functionals of the form
\begin{align}\label{surface_energy}
	u\mapsto \int_{J_u}g\big(u^-(x),u^+(x),\nu_u(x)\big)\dd \Hcal^{n-1}(x),
\end{align}
where $u$ is an $\rm SBV$-function with jump set $J_u$, jump normal $\nu_u$, and approximate limits $u^-$ and $u^+$ on both sides of $J_u$, and $g\colon \R^{n}\times\R^n\times\Scal^{n-1}\to [0,\infty)$ is a suitable energy density.
Such energies often appear in the context of fracture mechanics \cite{AmB95}, polycrystalline solids \cite{Car08, Car13, Car17}, liquid crystals \cite{AmB90i, AmB90ii}, free discontinuity problems \cite{AFP00}, or the relatively recent theory of structured deformations, see \cite{ChF97, DPO1993} or \cite{MMO2023} and the references therein.

While energies as in \eqref{surface_energy}, defined on the set of piecewise constant functions (in the sense of Caccioppoli), were first addressed in \cite{Alm76}, a general variational theory to handle existence of minimizers, relaxation, and $\Gamma$-convergence has been developed later in \cite{AmB90i,AmB90ii}.
For bounded densities, it was proven in \cite{AmB90ii, AFP00} that lower semicontinuity of the surface energy \eqref{surface_energy} is equivalent to \textit{$\rm BV$-ellipticity} of the corresponding density $g$. This notion is the surface-density-analogue of quasiconvexity, the key convexity notion in the bulk-case.
One calls $g$ $\rm BV$-elliptic if
\begin{align}\label{intro_BV}
	g(i,j,\eta) \leq \int_{J_u}g(u^-,u^+,\nu_u)\dd \Hcal^{n-1}
\end{align}
for every $(i,j,\eta)\in\R^n\times\R^n\times \Scal^{n-1}$ and every piecewise constant function $u$ on $Q_\eta$ with $\{u\neq u_{i,j,\eta}\}\Subset Q_\eta$;
here, the set $Q_\eta\subset \R^n$ describes an open unit cube with a face that is orthogonal to $\eta$ and $u_{i,j,\eta}$ is the elementary jump from $j$ to $i$ along the line $\{x\cdot \eta = 0\}$.

Motivated by the setting of structured deformations, in which the energies account for microscopic slips and separations and, generally, the direction in which they take place, we assume that $g$ has the shape 
\begin{align}\label{choice}
	g(i,j,\eta)\coloneqq f(i-j,\eta) \quad\text{ with } f(\alpha \lambda,\eta) = \alpha f(\lambda,\eta)\text{ for every } \alpha> 0,\,(\lambda,\eta)\in\R^n\times\Scal^{n-1},
\end{align}
subadditive and with linear growth in the first variable;
via \eqref{1hom_extension} below, the function $f$ can be viewed as positively $1$-homogeneous also in the second variable.
Since the pair $([u],\nu_u)$ with $[u]=u^+-u^-$ is only unique up to a sign, it is natural to require that $f$ is even, i.e., 
$f(-\lambda,-\eta) = f(\lambda,\eta)$ for every $(\lambda,\eta)\in\R^n\times\Scal^{n-1}$.

Our assumption \eqref{choice} on the surface density is, however, incompatible with boundedness, which is why only partial characterization results for lower semicontinuity are available.
It is straightforward to show that the proof of \cite[Theorem 5.14]{AFP00} can be modified without relying on boundedness. 
Hence, $\rm BV$-ellipticity is still necessary for the lower semicontinuity of the corresponding energy.
A partial sufficiency result, on the other hand, follows as in \cite[Corollary 2.5]{FPS21};
indeed, the $\rm BV$-ellipticity of the density yields lower semicontinuity of the energy along converging sequences of piecewise constant functions that are bounded in $L^{\infty}(\Omega;\R^n)$.

Since $\rm BV$-ellipticity is usually difficult to verify, one is interested in stronger notions that are easier to handle. Such concepts have been analyzed and compared extensively in the literature, for example, in \cite{AFP00} or \cite{AmB90ii, Car08, Car13, Car17}.
One such notion is \textit{biconvexity}, which requires that the surface density in \eqref{surface_energy} can be written as
$$g(i,j,\eta) = \Phi\big((j-i)\otimes\eta\big)\quad \text{ for every } (i,j,\eta)\in\R^n\times\R^n\times\Scal^{n-1},$$
with a convex, positively $1$-homogeneous function $\Phi\colon\R^{n\times n}\to [0,\infty)$. This property was introduced by Ambrosio \& Braides \cite{AmB90ii} in a finite-valued setting.
It turned out that biconvexity does indeed imply $\rm BV$-ellipticity \cite[Proposition 2.2]{AmB90ii}, but the reverse has only been conjectured.
Since every biconvex function is necessarily positively $1$-homogeneous in the first variable, this equivalence requires a type of $1$-homogeneity condition;
indeed, one can  easily construct a non-positively $1$-homogeneous $\rm BV$-elliptic function by exploiting joint convexity, see \cite[Definition 5.17, Theorem 5.20]{AFP00}.
The conjecture can thus only be true for densities of the form \eqref{choice}.
As proposed in \cite{AmB90ii}, the inequality 
\begin{align}\label{AmB90_inequality}
	g(i,j,\eta) \leq \sum_{k=1}^mg(i_k,j_k,\eta_k)\quad\text{ with } \sum_{k=1}^m(i_k-j_k)\otimes \eta_k = (i-j)\otimes\eta,
\end{align}
for all $(i,j,\eta),(i_k,j_k,\eta_k)\in\R^n\times\R^n\times\Scal^{n-1}$ and $m\in\N$, would verify that $\rm BV$-ellipticity reduces to biconvexity.
The estimate \eqref{AmB90_inequality} has been shown later in \cite{Sil17} by \v{S}ilhav\'y in the context of structured deformations, however, without establishing a connection to \cite{AmB90i,AmB90ii} or \cite{AFP00}.
In this paper, we merge the complementary results of the two communities and discuss different convexity and $\rm BV$-ellipticity notions. Our first contribution is the following equivalence:
\begin{theorem}[Characterization of $\rm \bm{BV}$-ellipticity]\label{theo:intro_BV}
	If $f\colon\R^n\times\Scal^{n-1}\to [0,\infty)$ is even and positively $1$-homogeneous in the first variable, then $f$ is $\rm BV$-elliptic if and only if $f$ is biconvex.
\end{theorem}
Note that the definition of the two properties of $f$ as above are canonically transferred from \eqref{choice}, see Definitions \ref{def:surface_qc} and \ref{def:biconvexity}.

Among the recent advances in the analysis of energies like \eqref{surface_energy} defined on piecewise rigid functions are \cite{FrS20, FPS21}.
In particular, Friedrich, Perugini \& Solombrino (cf.~\cite{FPS21}) carry the notions of $\rm BV$-ellipticity and biconvexity (as well as joint convexity) over to the ${\rm BD}$-setting of functions with bounded deformation.
They show for bounded densities that the energy functional \eqref{surface_energy}, defined on the set of piecewise rigid functions with skew-symmetric gradients
is lower semicontinuous if and only if $g$ is \textit{${\rm BD}$-elliptic}. 
The latter is similar to $\rm BV$-ellipticity in the sense that \eqref{intro_BV} holds for every $(i,j,\eta)\in\R^n\times\R^n\times \Scal^{n-1}$ and every piecewise rigid function $u$ with $\{u\neq u_{i,j,\eta}\}\Subset Q_\eta$. It is evident that ${\rm BD}$-elliptic functions are also $\rm BV$-elliptic.

In \cite{FPS21}, the authors also define the concept of \textit{symmetric biconvexity}, for which $g$ satisfies
\begin{align}\label{intro_symmetric_biconvex}
	g(i,j,\eta) = \Psi\big((j-i)\odot\eta\big)\quad \text{ for every } (i,j,\eta)\in\R^n\times\R^n\times\Scal^{n-1},
\end{align}
with a convex, positively $1$-homogeneous $\Psi\colon\R^{n\times n}_\skw\to[0,\infty)$; here $(i-j)\odot\eta$ is short for the symmetric part of $(i-j)\otimes\eta$.
Whereas \cite[Proposition 4.10]{FPS21} already establishes that symmetric biconvex functions with $\{\Psi=0\}=\{0\}$, where $\Psi$ is as in \eqref{intro_symmetric_biconvex}, are ${\rm BD}$-elliptic, the question whether the two notions are equivalent (under suitable conditions) remained open. 
Our second main result is the affirmation of this issue for the choice \eqref{choice}.
\begin{theorem}[Characterization of $\rm \bm{BD}$-ellipticity]\label{theo:intro_BD}
	If $f\colon\R^n\times\Scal^{n-1}\to [0,\infty)$ is even and positively $1$-homogeneous in the first variable, then $f$ is ${\rm BD}$-elliptic if and only if $f$ is symmetric biconvex.
\end{theorem}
Proving this equivalence involves several steps. 
We establish that symmetric biconvex functions are ${\rm BD}$-elliptic by providing an alternative proof of \cite[Proposition 4.10]{FPS21} that does not require the assumption $\{\Psi = 0\} = \{0\}$ with $\Psi$ as in \eqref{intro_symmetric_biconvex} by reorganizing results and arguments from \cite{FPS21}.
To obtain the reverse implication, we leverage the larger class of test functions, unveiling additional properties besides those inherited by their $\rm BV$-ellipticity. 
Precisely, we show that ${\rm BD}$-elliptic densities (or rather their positively $1$-homogeneous extensions, see \eqref{1hom_extension} below) are symmetric in the sense that their two arguments are interchangeable. This can be done by combining techniques from \cite{Sil17}, which are based on the positive $1$-homogeneity, classic arguments in \cite{AFP00}, and the class of piecewise rigid functions.
The final step is to a prove a symmetric analogue of inequality \eqref{AmB90_inequality} in the ${\rm BD}$-setting, for which we carefully adapt a construction in \cite{Sil17} from the $\rm BV$-setting.

\subsection*{Organization of this paper}
In Section \ref{sec:preliminaries}, we cover the notation used in this article as well as a few technical preliminaries. 
After that, we introduce and characterize a number of $\rm BV$-ellipticity notions, defined via different classes of test functions in \eqref{intro_BV}.
While some of these properties coincide with biconvexity, see Theorem \ref{theo:BV_elliptic}, we also highlight that others become trivial if the class of test functions is too large or small, see Propositions \ref{prop:trivial} and \ref{prop:J0_BV}. We briefly discuss an alternative approach to joint convexity and characterize the $\rm BV$-elliptic envelopes of functions of the form \eqref{choice}.

Section \ref{sec:BD} is then devoted to the ${\rm BD}$-setting, where we prove the equivalence of ${\rm BD}$-ellipticity and symmetric biconvexity in Theorem \ref{theo:BD_symmetric_biconvex}.
Similarly to before, we review the notion of symmetric joint convexity in our context of \eqref{choice} and provide characterizations of ${\rm BD}$-elliptic envelopes.
We round off the article with a curios example of a biconvex function that is symmetric biconvex although it does not appear to be so at first glance.

\section{Preliminaries}\label{sec:preliminaries}
\subsection{Notation}
Let $n\in\N$. We denote the standard basis vectors of $\R^n$ with $e_1,\ldots,e_n$. For the Euclidean scalar product of two vectors $a,b\in\R^n$, we write $a\cdot b$ and the length of $a$ is then given by $|a|=\sqrt{a\cdot a}$.
Their tensor product (or outer/ dyadic product) $a\otimes b\in\R^{n\times n}$ is defined componentwise as $(a\otimes b)_{ij} \coloneqq a_ib_j$ for every $i,j\in\{1,\ldots,n\}$;
we denote its symmetric part $\frac{1}{2}a\otimes b + \frac{1}{2}b\otimes a$ as $a\odot b$.
The $(n-1)$-dimensional unit sphere $\Scal^{n-1}$ consists of all vectors in $\R^n$ with unit length.
Let $\eta\in\Scal^{n-1}$ be given and let $\zeta_1,\ldots,\zeta_{n-1}\in\Scal^{n-1}$ be such that the matrix $S=(\eta|\zeta_1|\cdots|\zeta_{n-1})\in\R^{n\times n}$ satisfies $S^TS=SS^T=\Id$ and $\det S = 1$, where $(\cdot)^T$ stands for the transpose and $\Id\in\R^{n\times n}$ is the identity matrix.
With a little abuse of notation, we use the symbol $x\cdot\eta^\perp$ to indicate 
$x\cdot \zeta_i$ for every $i=1,\ldots,n-1$.
In particular, we write 
\begin{align}\label{perpendicular}
	-\alpha\leq x\cdot\eta^\perp\leq\alpha\qquad :\Longleftrightarrow\qquad -\alpha\leq x\cdot \zeta_i\leq\alpha,\qquad\text{for all $i\in\{1,\ldots,n-1\}$.}
\end{align}
for $\alpha\geq 0$.

The scalar product of two square matrices $A,B\in\R^{n\times n}$ shall be given as $A:B=\sum_{i,j=1}^nA_{ij}B_{ij}$; this scalar product then induces the Frobenius norm $|A| \coloneqq \sqrt{A:A}$ of $A$.
For the set of symmetric and skew-symmetric matrices in $\R^{n\times n}$ we write $\R^{n\times n}_{\sym}$ and $\R^{n\times n}_{\skw}$; note that $A:B=0$ if $A\in\R^{n\times n}_{\sym}$ and $B\in\R^{n\times n}_{\skw}$.

The notation $U\Subset V$ for two sets $U, V\subset \R^n$ means that $U$ is compactly contained in $V$.
Given $\eta\in \Scal^{n-1}$, we define $Q_\eta$ as the open unit cube in $\R^{n}$  centered in the origin such that two faces are orthogonal to $\eta$. 

Moreover, we define $u_{\lambda, \eta}= \lambda \mathbbm{1}_{\{x\cdot \eta \geq 0\}}$ on $Q_\eta$ as the elementary jump of $\lambda\in\R^n$ accross the midplane of $Q_\eta$ perpendicular to $\eta$; here,
$\mathbbm{1}_U$ is the indicator function of set $U\subset \R^n$, which is $1$ on $U$ and vanishes on $\R^n\setminus U$.
A function $h\colon\R^n\to\R$ is called positively $1$-homogeneous if $h(\alpha\xi)=\alpha h(\xi)$ for all $\xi\in \R^n$ and all $\alpha>0$.
We say that a function $f\colon \R^n\times \Scal^{n-1}\to \R$ is even if $f(-\lambda, -\eta) = f(\lambda, \eta)$ for all $\lambda\in \R^n$ and $\eta\in \Scal^{n-1}$. For such a function $f$, we introduce its positively $1$-homogeneous extension in the second variable as
\begin{align}\label{1hom_extension}
\bar{f}\colon \R^n\times \R^n\to \R,\ (\lambda,\zeta)\mapsto \begin{cases}
\displaystyle |\eta| f\Big(\lambda, \frac{\eta}{|\eta|}\Big) & \text{for $\eta\in \R^n\setminus\{0\}$,}\\
0 & \text{for $\eta=0$.}
\end{cases}
\end{align}
By $\Hcal^{n-1}$ we mean the $(n-1)$-dimensional Hausdorff measure and $\Lcal^n$ is the Lebesgue measure in $\R^n$.

Let $U\subset \R^n$ be measurable and $1\leq p \leq \infty$; then we employ the standard notation for the Lebesgue spaces $L^p(U;\R^n)$ and the spaces ${\rm BV}(U;\R^n), {\rm SBV}(U;\R^n)$, as well as $C^1(U;\R^n)$. If $u\in {\rm BV}(U;\R^n)$, then we write $J_u$ for the jump set of $u$, $\nu_u\in\Scal^{n-1}$ for its normal, 
and $[u]\coloneqq u^+ - u^-$, where $u^+$ and $u^-$ are the approximate limits on both sides of $J_u$; note that the pair $([u],\nu_u)$ is only unique up to a sign, which is why we always work with even surface densities.

\subsection{Auxiliary results}

We first prove that rank-one matrices have a decomposition into tensor products of two vectors in $\R^n$ and $\Scal^{n-1}$ that is unique up to a sign.
\begin{lemma}\label{lem:symmetric_tensors}
	If $(\lambda,\eta),(\lambda',\eta')\in\R^{n}\times\Scal^{n-1}$ with $\lambda,\lambda' \neq 0$ satisfy
	$\lambda\otimes\eta = \lambda'\otimes\eta'$, then it holds that $(\lambda,\eta)=(\lambda',\eta')$ or $(\lambda,\eta)=(-\lambda',-\eta')$.
\end{lemma}
\begin{proof}
	Choose any $\xi\in\R^n$ such that $\eta\cdot\xi= 0$, then it holds that
	$$(\eta' \cdot\xi)\lambda' = (\lambda'\otimes \eta')\xi =(\lambda\otimes \eta)\xi  = (\eta\cdot\xi)\lambda = 0.$$
	Since $\lambda' \neq 0$, it holds that $\eta' \cdot \xi = 0$ for any $\xi\in\R^n$ with $\xi\cdot\eta=0$ and thus $\eta'$ is a multiple of $\eta$, which results in either
	$\eta' = \eta$ or $\eta' = - \eta$ because both vectors are normalized.
	Then, for any $x\in\R^n$ we find that either
	$$\big(x\cdot(\lambda - \lambda')\big)\eta = 0 \quad\text{or}\quad \big(x\cdot(\lambda + \lambda')\big)\eta = 0,$$
	which means that $\lambda=\lambda'$ or $\lambda=-\lambda'$.
\end{proof}

Next, we briefly cover a few properties of one of the central functions in \cite[Theorem 2.3]{Sil17}, which will also be relevant in this work.

\begin{lemma}\label{lem:Phi_f}
	Let $f\colon \R^n\times \Scal^{n-1}\to [0,\infty)$ be even, positively $1$-homogeneous in the first variable, and continuous. Then, the function $\Phi_f \colon \R^{n\times n}\to [0,\infty)$ given by
	\begin{align}\label{Phi1}
		\begin{split}
			\Phi_f(F) &\coloneqq \inf\Big\{\sum_{i=1}^{m} f(\lambda_i,\eta_i) : m\in\N,\, (\lambda_i,\eta_i)\in \R^n\times\Scal^{n-1} \text{ for all } i\in\{1,\ldots,m\} \\
					&\hspace{9cm}\text{ with } \sum_{i=1}^m\lambda_i\otimes \eta_i =F\Big\}
		\end{split}
	\end{align}
	for $F\in\R^{n\times n}$, is positively $1$-homogeneous and convex and the integer $m$ in \eqref{Phi1} can be chosen as $m=n^2+1$.
\end{lemma}
\begin{proof} 
	\textit{Step 1: Auxiliary function $\phi_f$.} First, we define the function
	\begin{align}\label{phi_extension}
	  	\phi_f\colon \R^{n\times n}\to[0,\infty),\ F\to \begin{cases}f(\lambda,\eta) &\text{ if } F = \lambda\otimes \eta \text{ for } (\lambda,\eta)\in\R^{n}\times\Scal^{n-1},\\ \infty &\text{ otherwise,}\end{cases}
	\end{align}
	and show that $\phi_f$ is well-defined.
	If $F=0$, then $F=0\times \eta$ for any $\eta\in\Scal^{n-1}$. Since $f$ is continuous and positively $1$-homogeneous in the first variable, it holds that $f(0,\eta)=0$. The case $F\neq 0$ can be handled via Lemma \ref{lem:symmetric_tensors} and the evenness of $f$.
	
	\textit{Step 2: Convex envelope of $\phi_f$.}
	As $f$ is positively $1$-homogeneous in the first variable the function $\phi_f$ is positively $1$-homogeneous.
	Its convex envelope	is then also positively $1$-homogeneous and coincides with $\Phi_f$, since
	  \begin{align*}
	  	&\inf\Big\{\sum_{i=1}^{n^2+1} \mu_i \phi_f(F_i) : \mu_i\geq 0, F_i\in\R^{n\times n} \text{ for all } i\in\{1,\ldots,n^2+1\}\\
		&\hspace{7cm}	 \text{ with } \sum_{i=1}^{n^2+1}\mu_i = 1 \text{ and }\sum_{i=1}^{n^2+1}\mu_iF_i = F\Big\}\\	  	
	  	&=\inf\Big\{\sum_{i=1}^{n^2+1} \mu_i f(\lambda_i,\eta_i) : \mu_i\geq 0, (\lambda_i,\eta_i)\in\R^n\times\Scal^{n-1} \text{ for all } i\in\{1,\ldots,n^2+1\}\\
		&\hspace{7cm}	 \text{ with } \sum_{i=1}^{n^2+1}\mu_i = 1 \text{ and }\sum_{i=1}^{n^2+1}\mu_i\lambda_i\otimes\eta_i = F\Big\}\\
		&=\inf\Big\{\sum_{i=1}^{n^2+1} f(\lambda_i,\eta_i) : (\lambda_i,\eta_i)\in\R^n\times\Scal^{n-1} \text{ for all } i\in\{1,\ldots,n^2+1\}\text{ with }\sum_{i=1}^{n^2+1}\lambda_i\otimes\eta_i =F\Big\}\\
		&=\inf\Big\{\sum_{i=1}^{m} f(\lambda_i,\eta_i) : m\in\N,\, (\lambda_i,\eta_i)\in \R^n\times\Scal^{n-1} \text{ for all } i\in\{1,\ldots,m\} \text{ with } \sum_{i=1}^m\lambda_i\otimes \eta_i =F\Big\}.
	  \end{align*}
	  The last two equalities are a direct consequence of the positive $1$-homogeneity of $f$ in the first variable and the proof of \cite[Theorem 2.35]{Dac08}.
\end{proof} 

\begin{remark}[Alternative representations of $\bm{\Phi_f}$]\label{rem:Phi_f_representations}
	Let $f\colon\R^n\times \Scal^{n-1}\to [0,\infty)$ be even, positively $1$-homogeneous and subadditive in the first variable, and continuous.
	In particular, $f$ satisfies
	\begin{align}\label{linear_bound}
		f(\lambda,\eta) \leq C|\lambda| \quad\text{for every $(\lambda,\eta)\in\R^n\times\Scal^{n-1}$}
	\end{align}
	with $C=\max_{\Scal^{n-1}\times\Scal^{n-1}} f$.
	According to \cite[Theorem 2.3]{Sil17}, the function $\Phi_f \colon \R^{n\times n}\to [0,\infty)$ in \eqref{Phi1} can then alternatively be expressed as
	\begin{align*}
		&\Phi_f(F) = \sup\big\{\theta(F) : \theta \text{ is subadditive on } \R^{n\times n},\, \theta(\lambda\otimes\eta)\leq f(\lambda,\eta) \text{ for all } (\lambda,\eta)\in\R^n\times\Scal^{n-1} \big\}\\[1mm]
		&=\inf\Big\{\int_{J_u}f([u],\nu_u)\dd\Hcal^{n-1} : u\in {\rm SBV}(Q_\eta;\R^n),\, \nabla u = 0 \text{ on } Q_\eta,\, u(x) = Fx \text{ for }x\in\partial Q_\eta\Big\}\\[1mm]
		&=\inf\Big\{\int_{J_u}f([u],\nu_u)\dd\Hcal^{n-1} : u\in {\rm SBV}(Q_\eta;\R^n),\, \int_{Q_\eta}\nabla u \dd x =0,\, u(x) = Fx \text{ for }x\in\partial Q_\eta\Big\}
	\end{align*}
	for every $F\in\R^{n\times n}$. Moreover, it holds that
	\begin{align*}
		\Phi_f(\lambda\otimes \eta) = \inf\Big\{\int_{J_u}f([u],\nu_u)\dd\Hcal^{n-1} : u\in {\rm SBV}(Q_\eta;\R^n),\, u=u_{\lambda,\eta} \text{ on }\partial Q_\eta,\,\nabla u = 0 \text{ on } Q_\eta\Big\}
	\end{align*}
	for every $(\lambda,\eta)\in\R^n\times\Scal^{n-1}$.
\end{remark}

Finally, we state the well-known fact that convexity and subadditivity are equivalent for positively $1$-homogeneous function. This result will be needed a few times in Section \ref{sec:BV}.

\begin{lemma}\label{lem:subadditivity}
Let $h\colon\R^n\to \R$ be positively $1$-homogeneous, then $h$ is subadditive if and only if $h$ is convex.
In this case, the function $h$ is also continuous.
\end{lemma}

\section{$\rm BV$-ellipticity and related notions}\label{sec:BV}
\subsection{Basic definitions and properties}
First, we provide the reader with a few classes of functions that appear in the literature, though often without explicit notation, in the context of ${\rm BV}$-ellipticity and lower semicontinuity of surface energy functionals. For $\eta\in\Scal^{n-1}$, we introduce
\begin{align}\label{all_sets}
\SBV= & \{\varphi\in {\rm SBV}(Q_\eta;\R^n): \varphi=0 \text{ on $\partial Q_\eta$}\},\nonumber\\[.5em]
\Av= & \bigg\{\varphi\in {\rm SBV}(Q_\eta;\R^n):  \int_{Q_\eta}\nabla \varphi \dd x=0, \varphi=0 \text{ on $\partial Q_\eta$}\bigg\},\nonumber\\[.5em]
\PC= & \{\varphi\in {\rm SBV}(Q_\eta;\R^n): \nabla \varphi=0 \text{ on $Q_\eta$}, \varphi=0 \text{ on $\partial Q_\eta$}\},\nonumber\\[.5em]
\Ca= &\bigg\{\varphi \in {\rm SBV}(Q_\eta;\R^n): \varphi= \sum_{k\in\N}\lambda_k \mathbbm{1}_{P_k} \text{ with $\lambda_k\in \R^n$},\nonumber\\ 
	&\phantom{u \in {\rm SBV}(Q_\eta;\R^n):}\text{and $(P_k)_k$ a Caccioppoli partition and $\text{supp\,}\varphi \Subset Q_\eta$}\bigg\},\nonumber\\[.5em]
\J= & \{\varphi\in {\rm SBV}(Q_\eta;\R^n): \varphi=\lambda \mathbbm{1}_{P} \text{ with $\lambda\in \R^n$, $P\Subset Q_\eta$}\}.
\end{align}
The space $\SBV$ appears in \cite[Theorem 4.2.2]{BFM98}, the subset $\Av$ with vanishing average of the gradient and $\PC$ with vanishing gradient can be found in \cite[Theorem 2.16, Theorem 2.17]{ChF97}. The set of piecewise constant functions $\Ca$ (in the sense of Caccioppoli) is the standard class of test functions for ${\rm BV}$-ellipticity and appears in, for instance, \cite[Definition 5.13]{AFP00}. Single jumps in $\J$ are an addition of ours to round of the discussion about different ${\rm BV}$-ellipticity notions.

Clearly, it holds that
\begin{align}\label{sets}
\J \subset \Ca \subset \PC \subset \Av \subset \SBV.
\end{align}
All inclusions above are also strict: while the first and last one are obvious, the other other two might not be as easy to see.
For the second inclusion, we refer to the construction in \cite[Lemma 5.2]{Sil17}.
As for $\PC\subsetneq \Av$, we choose $A\in\R^{n\times n}\setminus\{0\}$ and $\eta\in\Scal^{n-1}$, define
\begin{align}\label{Qpm}
	Q_\eta^+  \coloneqq  \big\{x\in Q_\eta: x\cdot\eta\geq 0\big\}\quad\text{and}\quad  Q_\eta^-  \coloneqq  \big\{x\in Q_\eta: x\cdot\eta<0\big\},
\end{align}
and
$$\varphi(x)\coloneqq \begin{cases}
0 & \text{if $\displaystyle x\in Q_\eta\setminus\tfrac12Q_\eta$,}\\
\pm Ax & \text{if $x\in \frac12 Q_\eta^\pm$,}
\end{cases}\quad x\in Q_\eta,$$
and observe that $\varphi\in \Av\setminus \PC$.

Now, we introduce several ${\rm BV}$-ellipticity notions with varying classes of test functions.
\begin{definition}[$\rm \bm{BV}$-ellipticity]\label{def:surface_qc}
	Let $f\colon\R^n\times \Scal^{n-1}\to [0,\infty)$ be an even function.
	We say that $f$ is ${\rm BV}$-elliptic if for any $\lambda\in \R^n$ and $\eta\in \Scal^{n-1}$ it holds that
	\begin{align}\label{surface_qc}
		f(\lambda, \eta) \leq \int_{J_u} f\big([u], \nu_u\big) \dd \Hcal^{n-1}
	\end{align}
	for all $u\in u_{\lambda, \eta} + \Ca$. 
	
	More generally, we say that $f$ is ${\rm BV}$-elliptic with respect to $\SBVs$/ $\PCs$/ $\Avs$/ $\Js$ when \eqref{surface_qc} holds with
	$\Ca$ replaced by one of the corresponding sets introduced in \eqref{all_sets}.
\end{definition}

In Section \ref{subsec:BV_char}, we provide characterizations of all these ${\rm BV}$-ellipticity notions. 
While ${\rm BV}$-ellipticity with respect to the extreme cases $\Js$ and $\SBVs$ result in trivial statements, ${\rm BV}$-ellipticity with respect to $\Cas$, $\PCs$, and $\Avs$ are all equivalent and coincide, under suitable conditions, with biconvexity. In general, the latter is a stronger notion and is much easier to verify. We introduce this concept in the next definition and adapt it from \cite[Section 2.2]{AmB90ii}.

\begin{definition}[Biconvexity]\label{def:biconvexity}
	We call $f\colon \R^n\times \Scal^{n-1}\to [0,\infty)$ biconvex if there exists a convex, positively $1$-homogeneous function $\Phi\colon \R^{n\times n}\to [0,\infty)$ such that 
	$$f(\lambda,\eta) = \Phi\big(\lambda\otimes \eta\big)$$
	for every $(\lambda,\eta)\in\R^n\times \Scal^{n-1}$.
\end{definition}

Naturally, any biconvex map has to be positively $1$-homogeneous in the first argument.
It was already conjectured in \cite[Page 9]{AmB90ii} that ${\rm BV}$-ellipticity with respect to $\Cas$ and biconvexity are equivalent concepts.
This conjecture can, however, not be true in general. 
To see this, one needs to construct a ${\rm BV}$-elliptic function that is not a positively $1$-homogeneous in the first variable.
This can be done with the help of a jointly convex one with said property, cf.~Remark \ref{rem:joint_convexity} or \cite[Definition 5.17]{AFP00}, as every jointly convex function is ${\rm BV}$-elliptic with respect to $\Cas$ in view of \cite[Theorem 5.20]{AFP00}.
If one amends this question by requiring positive $1$-homogeneity in the first variable, then the two definitions are indeed equivalent, as is shown in Theorem \ref{theo:BV_elliptic} below.

We now gather a few properties emanating from ${\rm BV}$-ellipticity with respect to $\Js$. 
This statement can be drawn from combining and adapting the proofs of Theorem~5.11 and Theorem~5.14 in \cite{AFP00}. 
The benefit of our argument is a direct proof which does not rely on lower semicontinuity arguments.

\begin{proposition}\label{prop:subadditivity}
If $f \colon \R^{n}\times \Scal^{n-1}\to[0,\infty)$ is even and ${\rm BV}$-elliptic with respect to $\Js$, then the following statements hold true.
\begin{itemize}
\item[$i)$] For any $\eta\in \Scal^{n-1}$, the function $f(\cdot, \eta)\colon\R^n\to [0,\infty)$ is subadditive.  
\item[$ii)$] For any $\lambda\in\R^n$, the function $\bar{f}(\lambda, \cdot) \colon \R^n\to [0,\infty)$ (cf.~\eqref{1hom_extension}) is convex.
\end{itemize}
\end{proposition}
\begin{proof}
$i)$ For fixed $\lambda, \xi\in \R^n$ and $\eta\in \Scal^{n-1}$, we use the single-jump test functions $(\varphi_k)_k\subset \J$ defined by
\begin{align*}
	\varphi_k = -\xi\mathbbm{1}_{\{0\leq x\cdot \eta \leq \frac{1}{k}, -\frac{1}{2}+\frac{1}{2k} \leq x\cdot \eta^\perp \leq \frac{1}{2} - \frac{1}{2k}\}}, 	\qquad \text{for $k>2$,}
\end{align*}
see Figure~\ref{fig:one}(a); recall also the notation \eqref{perpendicular}. 
Setting $u_k\coloneqq u_{\lambda,\eta}+\varphi_k$ (see Figure~\ref{fig:one}(b)), the ${\rm BV}$-ellipticity with respect to $\J$ gives
\begin{equation*}
\begin{split}
f(\lambda,\eta)\leq & \int_{J_{u_k}} f([u_k], \nu_{u_k})\dd{\Hcal^{n-1}} \\
= & \Big(1-\frac1k\Big)^{n-1} \big(f(\lambda-\xi,\eta)+f(\xi,\eta)\big) \\ & \qquad \qquad + \frac1k \Big(1-\frac1k\Big)^{n-2} \bigg(2f(\lambda,\eta)+\sum_{i=1}^{n-1} \big(f(-\xi,\zeta_i)+f(-\xi,-\zeta_i)\big)\bigg);
\end{split}
\end{equation*}
letting $k\to \infty$ yields
\begin{align*}
	f(\lambda, \eta) \leq f(\lambda -\xi, \eta) + f(\xi, \eta).
\end{align*}
Substituting $\lambda$ by $\lambda+\xi$ implies
\begin{align*}
	f(\lambda + \xi, \eta)\leq f(\lambda, \eta) + f(\xi, \eta)  \qquad \text{for $\lambda, \xi\in \R^{n}$,}
\end{align*}
which means that $f(\cdot, \eta)$ is subadditive for any fixed $\eta\in \Scal^{n-1}$. 

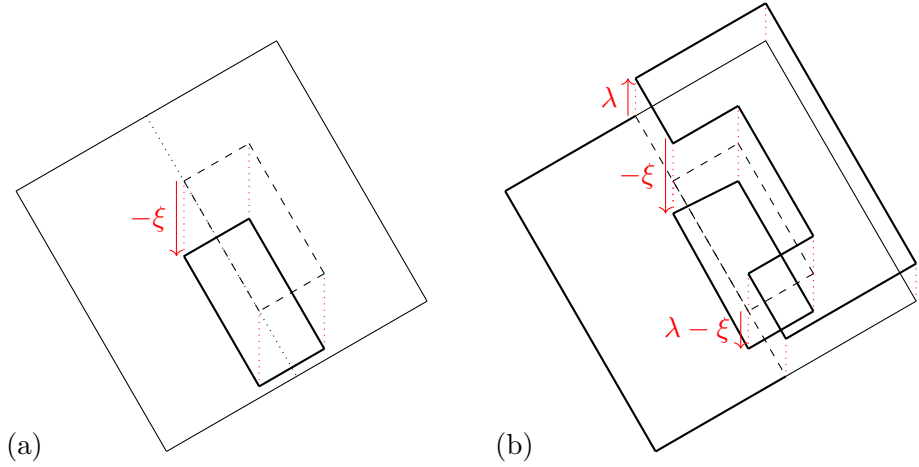
\begin{figure}[!ht]
\begin{center}
\begin{tikzpicture}
\node at (-3,-3) [text=black,above right] {(a)};
\draw[color=black] (.732,2.732) -- (-2.732,.732);
\draw[color=black] (-.732,-2.732) -- (-2.732,.732);
\draw[color=black] (-.732,-2.732) -- (2.732,-.732);
\draw[color=black] (.732,2.732) -- (2.732,-.732);
\draw[color=black,dotted] (-1,1.732) -- (1,-1.732);
\draw[color=black,dashed] (1.366,-.366) -- (.366,1.366);
\draw[color=black,dashed] (-.5,.866) -- (.366,1.366);
\draw[color=black,dashed] (-.5,.866) -- (.5,-.866);
\draw[color=black,dashed] (1.366,-.366) -- (.5,-.866);
\begin{scope}[shift={(0,-1)}]
\draw[color=red,dotted] (1.366,-.366) -- (1.366,.634);
\draw[color=red,dotted] (-.5,.866) -- (-.5,1.866);
\draw[color=red,dotted] (.5,-.866) -- (.5,.134);
\draw[color=red,dotted] (.366,1.366) -- (.366,2.366);
\draw[color=red,<-] (-.6,.866) --++ (0,0.5) node [left] {$-\xi$} --++ (0,0.5);
\end{scope}
\begin{scope}[shift={(0,-2)}]
\draw[color=black,thick] (1.366,.634) -- (.366,2.366);
\draw[color=black,thick] (.366,2.366) -- (-.5,1.866);
\draw[color=black,thick] (-.5,1.866) -- (.5,.134);
\draw[color=black,thick] (.5,.134) -- (1.366,.634);
\end{scope}
\end{tikzpicture}\qquad
\begin{tikzpicture}
\node at (-3,-3) [text=black,above right] {(b)};
\draw[color=black] (.732,2.732) -- (-1,1.732);
\draw[color=black,thick] (-1,1.732) -- (-2.732,.732);
\draw[color=black,thick] (-.732,-2.732) -- (-2.732,.732);
\draw[color=black,thick] (-.732,-2.732) -- (1,-1.732);
\draw[color=black] (2.732,-.732) -- (1,-1.732);
\draw[color=black] (.732,2.732) -- (2.732,-.732);
\draw[color=black,dashed] (-1,1.732) -- (1,-1.732);
\draw[color=black,dashed] (1.366,-.366) -- (.366,1.366);
\draw[color=black,dashed] (-.5,.866) -- (.366,1.366);
\draw[color=black,dashed] (-.5,.866) -- (.5,-.866);
\draw[color=black,dashed] (1.366,-.366) -- (.5,-.866);
\draw[color=red,dotted] (1.366,-.366) -- (1.366,-.866);
\draw[color=red,dotted] (-.5,.866) -- (-.5,.434);
\draw[color=red,dotted] (.5,-.866) -- (.5,-1.366);
\draw[color=red,dotted] (.366,1.366) -- (.366,.866);
\draw[color=black,thick] (1.366,-.866) -- (.366,.866);
\draw[color=black,thick] (.366,.866) -- (-.5,.434);
\draw[color=black,thick] (-.5,.434) -- (.5,-1.366);
\draw[color=black,thick] (.5,-1.366) -- (1.366,-.866);
\draw[color=red,dotted] (-1,1.732) -- (-1,2.232);
\draw[color=red,dotted] (1,-1.232) -- (1,-1.732);
\draw[color=red,dotted] (.732,2.732) -- (.732,3.232);
\draw[color=red,dotted] (2.732,-.732) -- (2.732,-.232);
\draw[color=red,dotted] (-.5,1.366) -- (-.5,.434);
\draw[color=red,dotted] (1.366,-.366) -- (1.366,.134);
\draw[color=red,dotted] (.5,-.866) -- (.5,-.366);
\draw[color=red,dotted] (.366,1.366) -- (.366,1.866);
\draw[color=red,->] (-1.1,1.732) --++(0,.25) node [left] {$\lambda$} --++ (0,.25);
\draw[color=red,<-] (-.6,.434) --++(0,.5) node [left] {$-\xi$} --++ (0,.5);
\draw[color=red,<-] (.4,-1.366) --++(0,.25) node [left] {$\lambda-\xi$} --++ (0,.25);
\draw[color=black,thick] (-1,2.232) -- (-.5,1.366);
\draw[color=black,thick] (-.5,1.366) -- (.366,1.866);
\draw[color=black,thick] (.366,1.866) -- (1.366,.134);
\draw[color=black,thick] (1.366,.134) -- (.5,-.366);
\draw[color=black,thick] (.5,-.366) -- (1,-1.232);
\draw[color=black,thick] (1,-1.232) -- (2.732,-.232);
\draw[color=black,thick] (2.732,-.232) -- (.732,3.232);
\draw[color=black,thick] (.732,3.232) -- (-1,2.232);
\end{tikzpicture}
\captionof{figure}{(a) the function $\varphi_k$ and (b) the function $u_k$ in dimension $n=2$ (here pictured for $k=4$), the dashed lines marking the jump set $J_{u_k}$. For the purpose of illustration, here we have taken $\lambda,\xi\in\R^2$ with vanishing first component.}
\label{fig:one}
\end{center}
\end{figure}
\smallskip

$ii)$  In view of Lemma~\ref{lem:subadditivity}, we can equivalently show that $\eta\mapsto \bar{f}(\lambda,\eta)$ is subadditive on $\R^n$.
We illustrate the proof in the case $n=2$ for clarity and indicate how to modify it for general dimensions $n\in\N$.
 
Let $\eta_1,\eta_2\in\R^2$ and set $\eta_0\coloneqq\eta_1+\eta_2$ and $\tilde\eta_j\coloneqq\eta_j/|\eta_j|$ for every $j\in\{0,1,2\}$. 
On the ``upper'' side of the square $Q_{\tilde\eta_0}$ we build a triangle $T^0$ with side lengths $L_0=\rho>0$, $L_1=\rho|\eta_1|/|\eta_0|$, and $L_2=\rho|\eta_2|/|\eta_0|$ and normals $\tilde\eta_0$, $-\tilde\eta_1$, and $-\tilde\eta_2$, respectively. 
Notice that the triangle ``closes'' because of the definition of $\eta_0$.
Setting $\rho=\frac{1}{2k}$, we let $T^k$ be the corresponding triangle, and we denote by $T_k$ the union of $2k-2$ triangles given by shifting $T^k$ along the upper side of $Q_{\tilde\eta_0}$ so that $\dist(T_k,\partial Q_\eta)\geq\frac{1}{2k}$.
Let now $\varphi_k\in \J$ be the function taking the value $-\lambda$ on $T_k$, and use $u_k=u_{\lambda,\tilde\eta_0}+\varphi_k$ as a test function for the ${\rm BV}$-ellipticity of $f$ with respect to $\Js$. 
We have
\begin{equation}\label{2dcase}
\begin{split}
\bar{f}(\lambda,\eta_0)=&\, |\eta_0|f(\lambda,\tilde\eta_0)\leq|\eta_0|\int_{J_{u_k}} f([u_k],\nu_{u_k})\dd\Hcal^1 \\
=&\,|\eta_0|(2k-2)\bigg[f(\lambda,\tilde\eta_1)\frac{1}{2k}\frac{|\eta_1|}{|\eta_0|}+f(\lambda,\tilde\eta_2)\frac{1}{2k}\frac{|\eta_2|}{|\eta_0|}\bigg]\to \bar{f}(\lambda,\eta_1)+\bar{f}(\lambda,\eta_2)
\end{split}
\end{equation}
as $k\to \infty$, so that $f$ is subadditive.

To deduce the statement for arbitrary dimensions, it suffices to use the previous construction on a thin cuboid whose base is a $2$-dimensional section of the cube $Q_{\tilde\eta_0}$, and whose thickness is $\frac{1}{2k}$ in the remaining $n-2$ dimensions. 
In doing so, the bracket in \eqref{2dcase} must be modified to
$$\bigg(1-\frac1k\bigg)^{n-2}|\eta_0|(2k-2)\bigg[f(\lambda,\tilde\eta_1)\frac{1}{2k}\frac{|\eta_1|}{|\eta_0|}+f(\lambda,\tilde\eta_2)\frac{1}{2k}\frac{|\eta_2|}{|\eta_0|}\bigg]+\bigg(\frac{1}{2k}\bigg)^{n-2}|\eta_0|f(\xi,\bar\lambda),$$
which reduces to the right-most term in \eqref{2dcase} in the limit as $k\to\infty$.
\end{proof}

\begin{remark}\label{rem:sep_cvx}
Let $f\colon \R^{n}\times \Scal^{n-1}\to[0,\infty)$ be even, positively $1$-homogeneous in the first variable, and ${\rm BV}$-elliptic with respect to $\Js$, then
$\bar f$ is separately convex, in view of Lemma~\ref{lem:subadditivity}. 
In particular, it follows that $\bar{f}$, and thus also $f$, is continuous and satisfies \eqref{linear_bound}.
\end{remark}

\subsection{Characterization results}\label{subsec:BV_char}

In this section we show that, under suitable assumptions, ${\rm BV}$-ellipticity with respect to the extreme sets $\Js$ and $\SBVs$ turn out to be rather trivial and that all other ${\rm BV}$-ellipticity notions introduced in Definition \ref{def:surface_qc} coincide.

\begin{proposition}\label{prop:trivial}
	If $f\colon \R^n\times \Scal^{n-1}\to [0,\infty)$ is even and positively $1$-homogeneous in the first argument, then $f$ is ${\rm BV}$-elliptic with respect to $\SBVs$ if and only if $f=0$.
\end{proposition}
\begin{proof} 
For any $(\lambda,\eta)\in \R^n\times\Scal^{n-1}$ and $k\in\N$, let $\psi_k\in C^\infty_c(Q_\eta;\R^n)$ such that 
\begin{align*}
\psi_k=-\lambda\text{ in }(1-\tfrac{1}{k})Q_\eta \qand\sup_k \lVert \psi_k \rVert_\infty\leq C<+\infty,
\end{align*} 
for some positive constant $C$.
We define $\varphi_k \coloneqq \psi_k\mathbbm{1}_{Q_\eta^+}$, with $Q_{\eta}^+$ as in \eqref{Qpm} and observe that $\varphi_k\in \SBV$ by design. 
The only jumps of $u_k$ appear on the set
$$N_k \coloneqq \big\{x\in Q_\eta : x\cdot\eta = 0,\, -\tfrac{1}{2} + \tfrac{1}{2k} \leq x\cdot\eta^\perp \leq \tfrac{1}{2} - \tfrac{1}{2k}\big\}.$$
Since $f$ is ${\rm BV}$-elliptic with respect to $\SBVs$, one obtains with the test fields $u_k\coloneqq u_{\lambda,\eta}+\varphi_k$ and in view of Proposition~\ref{prop:subadditivity}\,$i)$ that 
\begin{align*}
	f(\lambda, \eta) & \leq \int_{J_{u_k}} f\big([u_k], \nu_{u_k}\big) \dd\Hcal^{n-1} = \int_{N_k} f\big(\lambda +\psi_k, \eta\big) \dd\Hcal^{n-1} \\ 
	&\leq \Hcal^{n-1}(N_k) f(\lambda, \eta) + C \int_{N_k\cap \{\psi_k\neq 0\}} \Big|f\bigg(\frac{\psi_k}{\lVert\psi_k\rVert_{L^\infty}}, \eta\bigg)\Big| \dd\Hcal^{n-1}\\ 
	&\leq (1+C)\max_{\xi\in \Scal^{n-1}} |f(\xi, \eta)| \Hcal^{n-1}(N_k) \to 0.
\end{align*}
The second step exploits that $f(0, \eta)=0$ for all $\eta\in \Scal^{n-1}$, noting that $f(\cdot, \eta)$ is continuous by Lemma~\ref{lem:subadditivity}, and the third one makes use of the subadditivity in combination with the positive $1$-homogeneity of $f$;
for the final two steps, we have used that $f$ is continuous in its first variable, and hence, $f(\cdot, \eta)$ is uniformly bounded on $\Scal^{n-1}$, and that 
$$\Hcal^{n-1} (N_k)\leq 1-(1-\tfrac{1}{k})^{n-1}$$ tends to zero as $k\to 0$.
This shows that $f\leq 0$, from which we conclude that $f=0$ since $f$ is non-negative by assumption. 
\end{proof}

The next theorem shows that separate convexity of the positively $1$-homogeneous extension of $f\colon\R^n\times\Scal^{n-1}\to[0,\infty)$ in the second variable is sufficient for the ${\rm BV}$-ellipticity of $f$ with respect to $\Js$.

\begin{proposition}\label{prop:J0_BV}
If $f\colon \R^n\times \Scal^{n-1}\to [0,\infty)$ is even and positively $1$-homogeneous in the first variable, then $f$ is separately convex if and only if $f$ is ${\rm BV}$-elliptic with respect to $\Js$.
\end{proposition}
\begin{proof}
The necessity follow immediately from Remark~\ref{rem:sep_cvx}. We now turn to the proof of sufficiency.
Let $(\lambda, \eta)\in \R^n\times \Scal^{n-1}$ and $\varphi\in \J$ with $\varphi=\xi \mathbbm{1}_P$ for $\xi\in \R^n$ and $P\Subset Q_\eta$ a set of finite perimeter. 
Moreover, take $P^+\coloneqq P\cap Q_\eta^+$ and $P^-\coloneqq P\cap Q_\eta^{-}$ with $Q_\eta^{\pm}$ as in \eqref{Qpm}. 
In the following, we take $\partial P$ as the reduced boundary of $P$ in the measure-theoretic sense and let $\nu_P\colon \partial P\to \Scal^{n-1}$ be the generalized outer normal to $P$ according to \cite[Section~3.5]{AFP00}; and analogously for $P^{\pm}$.
 
Setting $N_\eta\coloneqq\{x\in Q_\eta: x\cdot \eta=0\}$,
by the measure-theoretic Gau{\ss}-Green formula (see e.g.~\cite{AFP00}), 
\begin{align}\label{GaussGreen}
	\int_{\partial P^{\pm}\setminus N_\eta} \nu_{P^{\pm}}\dd\Hcal^{n-1} = - \int_{\partial P^{\pm}\cap N_\eta} \nu_{P^\pm} \dd{\Hcal^{n-1}}
 	=\pm \Hcal^{n-1}(\partial P^\pm \cap N_\eta)\eta.
\end{align}

By Jensen's inequality, exploiting the convexity of $\bar{f}$ in its second variable, we find together with~\eqref{GaussGreen} that 
\begin{subequations}\label{44}
\begin{align}\label{44a}
\begin{split}
		\int_{\partial P^+\setminus N_\eta}  f(-\xi, \nu_{P^+}) \dd{\Hcal^{n-1}}
   		&\geq \Hcal^{n-1}(\partial P^+\setminus N_\eta) f\Bigl(-\xi,  \dashint_{\partial P^+\setminus N_\eta}  \nu_{P^+}\dd{\Hcal^{n-1}}\Bigr) \\ 
   		& =  \Hcal^{n-1}(\partial P^+\setminus N_\eta) f(-\xi, \eta),
   	\end{split}
\end{align}
and similarly, 
\begin{align}
\int_{\partial P^-\setminus N_\eta}  f(-\xi, \nu_{P^-}) \dd{\Hcal^{n-1}} = \Hcal^{n-1}(\partial P^-\setminus N_\eta) f(-\xi, -\eta).
\end{align}
\end{subequations}

Further, let $u=u_{\lambda, \eta}+\varphi$ and invoke~\eqref{44} and the evenness of $f$ to obtain
\begin{align*}
	&\int_{J_u} f([u], \nu_u) \dd{\Hcal^{n-1}} =  \int_{\partial P^+\setminus N_\eta} f(-\xi, \nu_{P^+}) \dd{\Hcal^{n-1}} +  \int_{\partial P^-\setminus N_\eta} f(-\xi, \nu_{P^-}) \dd{\Hcal^{n-1}}\\ 
	&\qquad  +\int_{\partial P^+\setminus \partial P^- \cap N_\eta} f(-\xi-\lambda, -\eta) \dd{\Hcal^{n-1}} + \int_{\partial P^-\setminus \partial P^+ \cap N_\eta} f(\lambda-\xi, \eta) \dd{\Hcal^{n-1}}\\ 
	&\qquad +  \int_{N_\eta \setminus \partial P} f(\lambda, \eta) \dd{\Hcal^{n-1}}   + \int_{\partial P^+\cap \partial P^- \cap N_\eta} f(\lambda, \eta) \dd{\Hcal^{n-1}}\\ 
	&\geq \Hcal^{n-1} (\partial P^+\setminus N_\eta) f(-\xi, \eta) + \Hcal^{n-1} (\partial P^-\setminus N_\eta) f(\xi,\eta)\\
	&\qquad + \Hcal^{n-1}(\partial P^+\setminus \partial P^- \cap N_\eta) f(\xi + \lambda, \eta) +\Hcal^{n-1}(\partial P^-\setminus \partial P^+ \cap N_\eta) f(\lambda-\xi, \eta)\\ 
	&\qquad + \Hcal^{n-1} (N_\eta \setminus \partial P) f(\lambda, \eta)  + \Hcal(\partial P^+\cap\partial P^-\cap N_\eta) f(\lambda, \eta).
\end{align*}
Using the estimates
$$\Hcal^{n-1} (\partial P^+\setminus N_\eta)\geq \Hcal^{n-1}(\partial P^+\setminus \partial P^- \cap N_\eta)\quad\text{and}\quad f(\lambda, \eta) \leq f(\lambda \pm \xi,\eta) + f(\mp\xi, \eta),$$
which hold due to \eqref{GaussGreen} and Lemma \ref{prop:subadditivity} a), we find
\begin{align*}
	\int_{\Scal_u\cap Q_\eta} f([u], \nu_u) \dd{\Hcal^{n-1}} &\geq \Hcal^{n-1}(\partial P^+\setminus \partial P^- \cap N_\eta) \Big(f(\xi + \lambda, \eta) + f(-\xi, \eta)\Big) \\
	&\:+\Hcal^{n-1}(\partial P^-\setminus \partial P^+ \cap N_\eta)\Big(f(\lambda-\xi, \eta) + f(\xi,\eta)\Big)\\ 
	&\:+ \Hcal^{n-1} (N_\eta \setminus \partial P) f(\lambda, \eta)  + \Hcal(\partial P^+\cap\partial P^-\cap N_\eta) f(\lambda, \eta)	\geq f(\lambda,\eta),
\end{align*}
since
\begin{align*}
&\Hcal^{n-1}(N_\eta\setminus \partial P) + \Hcal^{n-1}(\partial P^+\setminus \partial P^-\cap N_\eta)+\Hcal^{n-1}(\partial P^-\setminus \partial P^+\cap N_\eta)  +  \Hcal(\partial P^+\cap\partial P^-\cap N_\eta)\\ & \qquad \qquad = \Hcal^{n-1}(N_\eta\setminus \partial P) + \Hcal^{n-1}(\partial P\setminus N_\eta) =\Hcal^{n-1}(N_\eta) =1.
\end{align*}

\end{proof}

As a consequence of~\cite[Lemma~6.2 and Lemma~6.3]{Sil17}, ${\rm BV}$-ellipticity can be characterized as follows:

\begin{theorem}[Characterization of $\rm \bm{BV}$-ellipticity]\label{theo:BV_elliptic}
Let $f\colon \R^m\times \Scal^{n-1}\to [0,\infty)$ be even and positively $1$-homogeneous in the first argument. 
Then the following statements are equivalent:
\begin{enumerate}
\item $f$ is ${\rm BV}$-elliptic with respect to $\Avs$;
\item $f$ is ${\rm BV}$-elliptic with respect to $\PCs$;
\item $f$ is ${\rm BV}$-elliptic with respect to $\Cas$;
\item $f$ is biconvex and $f(\lambda,\eta) = \Phi_f(\lambda\otimes \eta)$ for every $(\lambda,\eta)\in\R^{n}$, cf.~\eqref{Phi1}.
\end{enumerate}
\end{theorem}

\begin{proof}
The implications ``(1) $\Rightarrow$ (2)'', ``(2) $\Rightarrow$ (3)'' are trivial in light of \eqref{sets}. It remains to prove ``(3) $\Rightarrow$ (4)'' and ``(4) $\Rightarrow$ (1)''. Let $(\lambda,\eta)\in\R^n\times\Scal^{n-1}$ now be arbitrary.
\medskip

``(3) $\Rightarrow$ (4)''. As a consequence of Remark~\ref{rem:sep_cvx}, the function $f$ is separately convex, continuous and satisfies \eqref{linear_bound}. 
Then, the assumptions of Remark \ref{rem:Phi_f_representations} are satisfied in view of Lemma \ref{lem:subadditivity} and it holds that
\begin{align*}
	\Phi_f(\lambda\otimes \eta) &= \inf\bigg\{\int_{J_u} f([u], \nu_u)\, \dd\mathcal{H}^{n-1}:  u\in u_{\lambda, \eta}+ \PC\bigg\} \\
	&= \inf\bigg\{\int_{J_u} f([u], \nu_u)\, \dd\mathcal{H}^{n-1}:  u\in u_{\lambda, \eta}+ \Ca\bigg\};
\end{align*}
the latter equality follows from the proofs in~\cite[Section 6]{Sil17} where only test functions in $\Cas$ instead of $\PCs$ are employed.
Due to the ${\rm BV}$-ellipticty with respect to $\Cas$ of $f$, we then conclude that $\Phi_f(\lambda\otimes \eta)\geq f(\lambda, \eta)$. 
On the other hand, by choosing $u=u_{\lambda,\eta}$ as a test function, one immediately finds that $\Phi_f(\lambda\otimes \eta)\leq f(\lambda, \eta)$.
\medskip

``(4) $\Rightarrow$ (1)''. Let $\varphi\in \Av$ and $u=u_{\lambda, \eta}+\varphi$. 
Then, we infer by Jensen's inequality and by exploiting the positive $1$-homogeneity of $\Phi_f$ (see Lemma \ref{lem:Phi_f}) that
\begin{align*}
	\int_{J_u} f([u], \nu_u)\, \dd\mathcal{H}^{n-1}& = \int_{J_u} \Phi_f([u]\otimes \nu_u)\, \dd\mathcal{H}^{n-1} \geq \Phi_f\Bigl( \int_{J_u} [u] \otimes \nu_u\, \dd\mathcal{H}^{n-1}\Bigr) \\ & = \Phi_f(\lambda\otimes \eta) = f(\lambda, \eta).
\end{align*}
To see the equality before the last, we argue that
\begin{align*}
	\int_{J_u} [u]\otimes \nu_u\, \dd\mathcal{H}^{n-1}& = \int_{N_\eta}(\lambda +[\varphi])\otimes \eta\, \dd\mathcal{H}^{n-1} + \int_{J_u\setminus N_\eta} [\varphi]\otimes \nu_\varphi\, \dd\mathcal{H}^{n-1} \\ & =\lambda\otimes \eta + \int_{J_\varphi} [\varphi] \otimes \nu_\varphi\, \dd\mathcal{H}^{n-1} = \lambda\otimes \eta ,
 \end{align*}
 due to
\begin{align*}
	\int_{J_\varphi} [\varphi]\otimes \nu_\varphi \dd \Hcal^{n-1}  = D\varphi(Q_\eta) - \int_{Q_\eta} \nabla \varphi \dd{x} =  \int_{\partial Q_\eta} \varphi\otimes \nu_{Q_\eta} \dd{\Hcal^{n-1}} =0,
\end{align*}
where we have exploited that the mean value of $\nabla \varphi$ on $Q_\eta$ is zero and that $\varphi$ has zero boundary conditions on $\partial Q_\eta$. Moreover, we used the fine properties of $\rm SBV$ functions (see \cite{AFP00}) and  
\begin{align*}
	D\varphi(Q_\eta) = \int_{\partial Q_\eta} \varphi\otimes \nu_{Q_\eta} \dd{\Hcal^{n-1}},
\end{align*}
which follows from the trace theorem for ${\rm BV}$ functions, see e.g.~\cite[Section~5.3, Theorem~1]{EvG92}.
\end{proof}

\begin{remark}
It is known (see \cite[Proposition~4]{KrR16}) that $\rm SBV$-functions cannot generally be approximated by piecewise constant ones. Proposition~\ref{prop:trivial} and Theorem~\ref{theo:BV_elliptic} provide an additional confirmation of this fact: if such an approximation existed, then $\SBVs$-elliptic functions could be approximated by $\Cas$-elliptic ones, but this cannot be the case since the former class only contains the zero function.
\end{remark}

To close this section, we briefly discuss two closely related topics.
We first cover a relaxation result in the ${\rm BV}$-setting. Here, Theorem \ref{theo:BV_elliptic} is the key to characterizing ${\rm BV}$-elliptic envelopes of even functions $f\colon \R^n\times\Scal^{n-1}\to[0,\infty)$.
Two versions of such envelopes could be defined as follows:
\begin{align*}
	f^\BV(\lambda,\eta)& \coloneqq \sup\left\{h(\lambda,\eta) : h\text{ is ${\rm BV}$-elliptic and } h\leq f\right\},\\
	f_\BV(\lambda,\eta)& \coloneqq \inf\left\{\int_{J_u} f([u],\nu_u)\dd\Hcal^{n-1}: u \in u_{\lambda,\eta} + \Ca\right\}.
\end{align*}
In the following, we prove that they both suitably coincide with $\Phi_f$ under the assumptions in Theorem \ref{theo:BV_elliptic}.

\begin{proposition}[$\rm \bm{BV}$-elliptic envelope]\label{prop:BV_envelope}
	Let $f\colon \R^n\times \Scal^{n-1}\to [0,\infty)$ be even and positively $1$-homogeneous in the first variable.
	Then, it holds that
	$$f^\BV(\lambda,\eta) = f_\BV(\lambda,\eta) = \Phi_f(\lambda\otimes\eta),$$
	for every $(\lambda,\eta)\in\R^n\times\Scal^{n-1}$.
	In particular, the ${\rm BV}$-elliptic envelope of $f$ is ${\rm BV}$-elliptic, and $f$ is ${\rm BV}$-elliptic if and only if it coincides with its ${\rm BV}$-elliptic envelope.
\end{proposition}
\begin{proof}
	The function 
	$$\tilde{f}\colon \R^n\times\Scal^{n-1}\to[0,\infty),\quad (\lambda,\eta)\mapsto\Phi_f(\lambda\otimes\eta)$$
	is ${\rm BV}$-elliptic due to Theorem \ref{theo:BV_elliptic}. Hence, it follows
	$$\Phi_f(\lambda\otimes\eta)=\tilde{f}(\lambda,\eta)\leq f^\BV(\lambda,\eta)$$
	for every $(\lambda,\eta)\in\R^n\times\Scal^{n-1}$.
	The inequality $f_\BV\leq \tilde{f}$ has been shown in \cite[Lemma 6.2]{Sil17}, so it remains to show that $f^\BV\leq f_\BV$.
	To this end let $u\in u_{\lambda,\eta}+\Ca$ be any test function and let $\eps>0$. 
	By definition of $f^\BV$, we find a ${\rm BV}$-elliptic function $h$ with $h\leq f$ such that
	$$f^\BV(\lambda,\eta) \leq h(\lambda,\eta) + \eps \leq \int_{J_u}h([u],\nu_u)\dd\Hcal^{n-1} + \eps\leq \int_{J_u}f([u],\nu_u)\dd\Hcal^{n-1} + \eps.$$
	Since $u$ and $\eps$ are arbitrary, we conclude the remaining inequality.
\end{proof}

Lastly, we mention another closely related notion of convexity.
\begin{remark}[Joint convexity]\label{rem:joint_convexity}
	We say that an even function $f\colon \R^n\times\Scal^{n-1}\to[0,\infty)$ is \emph{jointly convex} if there exist Lipschitz continuous functions $g_i\in\C^1( \R^n;\R^n)$ for every $i\in\N$ such that
	\begin{align*}
		f(\lambda,\eta) = \sup_{i\in\N}\big(g_i(\lambda) - g_i(0)\big)\cdot \eta\quad\text{ for every } (\lambda,\eta)\in\R^n\times\Scal^{n-1}.
	\end{align*}
	
	This definition differs slightly from the literature \cite[Definition 5.17]{AFP00} in the sense that we do not require the functions $g_i$ to be defined on compact sets. We also do not necessitate uniform continuity and boundedness when extending the functions in \cite[Definition 5.17]{AFP00} to all of $\R^n$.
	\medskip
	
	It turns out that this convexity notion is also equivalent to ${\rm BV}$-ellipticity if $f$ is positively $1$-homogeneous in the first variable.
	Indeed, any jointly convex function is ${\rm BV}$-elliptic and the proof can be handled exactly as in \cite[Theorem 5.20]{AFP00}.
	
	On the other hand, any biconvex function is jointly convex since any convex function $\Phi\colon\R^{n\times n}\to[0,\infty)$ with $\Phi(0)=0$ can be approximated from below by linear functions $g_i \colon\R^{n\times n}\to\R$, which can be expressed as
	$g_i(F) = A_i : F$ for some $A_i\in\R^{n\times n}$ and every $F\in\R^{n\times n}$.
\end{remark}

\section{$\rm BD$-ellipticity and related notions}\label{sec:BD}
\subsection{Basic definitions}
First, we introduce the primary set of test functions relevant for ${\rm BD}$-ellipticity.
For $\eta\in\Scal^{n-1}$, we define the set of \emph{piecewise rigid} functions with compact support in $Q_\eta$ as
\begin{align*}
	\begin{split}
		\PR = \bigg\{&\varphi \in {\rm SBV}(Q_\eta;\R^n): \varphi(x) = \sum_{k\in\N} (A_kx +b_k)\one_{P_k}(x)\ \text{for } x\in Q_\eta, \\
		&A_k\in \R^{n\times n}_{\skw},\, b_k\in\R^n,\,\text{$(P_k)_k$ a Caccioppoli partition of $Q_\eta$, and $\supp \varphi\Subset Q_\eta$}\bigg\},
	\end{split}
\end{align*}
cf.~\cite[Section 2.1]{FPS21}, and give the following definition of ${\rm BD}$-ellipticity:

\begin{definition}[$\rm \bm{BD}$-ellipticity]
	We call an even function $f\colon\R^n\times \Scal^{n-1}\to [0,\infty)$ ${\rm BD}$-elliptic if
	\begin{align*}
		f(\lambda,\eta) \leq \int_{J_u} f\big([u],\nu_u\big) \dd \Hcal^{n-1}
	\end{align*}
	for every $u\in u_{\lambda,\eta} + \PR$.
\end{definition}

Several examples of ${\rm BD}$-elliptic functions can be found in Section \cite[Section 4]{FPS21}.

The next definition is the ${\rm BD}$-analogue of biconvexity, taken from \cite[Definition 4.8]{FPS21}, and tailored to our setting.

\begin{definition}[Symmetric biconvexity]\label{def:symmetric_biconvex}
	An even function $f\colon\R^n\times \Scal^{n-1}\to [0,\infty)$ is said to be symmetric biconvex, if there exists a convex and positively $1$-homogeneous function $\Psi:\R^{n\times n}_{\sym}\to [0,\infty)$ such that
	$$f(\lambda,\eta) = \Psi(\lambda\odot\eta)\quad\text{ for all } (\lambda,\eta)\in\R^n\times \Scal^{n-1}.$$
\end{definition}

Before we prove our second main theorem, we briefly summarize a few properties of ${\rm BD}$-elliptic functions as a direct consequence of their ${\rm BV}$-ellipticity.
\begin{remark}
	Let $f\colon\R^n\times \Scal^{n-1}\to [0,\infty)$ be even, positively $1$-homogeneous and ${\rm BD}$-elliptic. 
	In view of $\Ca\subset\PR$ and the chain of inclusions in \eqref{sets}, it follows from Proposition \ref{prop:subadditivity} and Remark \ref{rem:sep_cvx} that $\bar{f}$ as in \eqref{1hom_extension} is separately convex, continuous and satisfies \eqref{linear_bound}.
\end{remark}

\subsection{Characterization result}

First, we prove that ${\rm BD}$-elliptic functions that are positively $1$-homogeneous in the first variable are symmetric in the sense that the two variables can be switched.

\begin{proposition}[Symmetry of $\rm \bm{BD}$-elliptic functions]\label{prop:switch}
	If $f\colon\R^n\times \Scal^{n-1}\to [0,\infty)$ is even, positively $1$-homogeneous in the first variable, and ${\rm BD}$-elliptic, then
	\begin{align}\label{switch}
		\bar{f}(\lambda,\eta) = \bar{f}(\eta,\lambda) \quad\text{for all }(\lambda,\eta)\in\R^n\times \R^n,
	\end{align}
	with $\bar{f}$ as in \eqref{1hom_extension}.
\end{proposition}
\begin{proof}
	Since $\bar{f}$ is positively $1$-homogeneous in both variables, it suffices to establish \eqref{switch} for unit vectors.
	For $n=1$, the statement is clear, which is why first consider the case $n=2$. We detail the generalization to higher dimensions in Step 3.
	
	Without loss of generality, we may assume that $\lambda\in \Scal^{n-1}$ is arbitrary and $\eta = e_2$. The goal is then to show $$f(\lambda,e_2) \leq f(e_2,\lambda).$$
	Our proof strategy relies on the construction of suitable test functions that generate many small jumps by multiples of $e_2$ in the direction $\lambda$, and at the same time compensate the elementary jump $u_{\lambda,e_2}$.
	To establish the above inequality for any given $\eta \in \Scal^{n-1}$, we simply rotate the construction.	
	The desired equality \eqref{switch} can be obtained by exchanging the roles of $\lambda$ and $\eta$.
	If $\lambda=\pm e_2$, there is nothing to prove, which is why we exclude this case in the following calculations.
	\medskip
	
	\textit{Step 1: Construction on a single triangle ($n=2$).}
	Let $k\in\N$ with $k>2$. In this step, we consider the unique matrix $A_k\in\R^{2\times 2}_{\skw}$ with the property that
	\begin{align}\label{Ak}
		A_ke_1 = -k^2e_2.
	\end{align}
	We define a small triangle $\Delta_k$ with vertices $0$, $\frac{1}{k}e_1$ and $\xi_k \coloneqq \frac{1}{k}e_1 + \frac{2}{k^2}w$, where $w\in\Scal^1$ is the vector with 
	\begin{align}\label{Akw}
		\frac{1}{k^2}A_kw=\lambda.
	\end{align}
	For now, we assume that $w\cdot e_2 >0$; the case $w\cdot e_2<0$ is detailed at the end of this step.
	
	We intersected the triangle by $k^2-1$ parallel lines with normal $\lambda$ in such a way that their intersection points with $[0,\frac{1}{k})\times \{0\}$ are equidistant with distance $\frac{1}{k^3}$. The resulting geometric figures inside $\Delta_k$ (a single triangle and $k^2-1$ trapezoids) are denoted by $M_k^1,\ldots,M_k^{k^2}$, counting from left to right.
	For an illustration of the geometric setup, see Figure \ref{fig:triangle0}.
	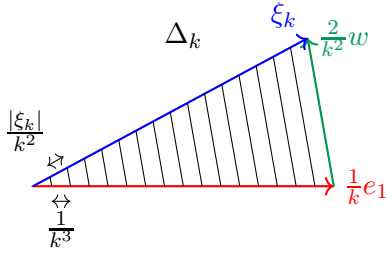
\begin{figure}[h!]
		\begin{tikzpicture}[scale=8]
			\draw (0,0) -- (1/2,0) --++ (100:1/4) -- cycle;
			\draw (1/4,1/4) node {$\Delta_k$};
			\draw [red,thick,->] (0,0) -- (1/2,0) node [right] {$\frac{1}{k}e_1$};
			\draw [ForestGreen,thick,->] (1/2,0) --++ (100:1/4) node [right] {$\frac{2}{k^2}w$};
			\draw [blue,thick,->] (0,0) -- ($(1/2,0) + (100:1/4)$) node [above left] {$\xi_k$};
			\draw [<->] (1/32,-.025) --++ (1/64,0) node [below] {$\frac{1}{k^3}$} --++ (1/64,0);
			\draw [<->] ($(0,0.02)+(30:1/36)$) --++ (30:1/64) node [above left] {$\frac{|\xi_k|}{k^2}$} --++ (30:1/64);
			\begin{scope}
				\clip (0,0) -- (1/2,0) --++ (100:1/4) -- cycle;
				\foreach \j in {0,...,15}{
					\draw (\j/32,0) --++ (100:1/4);
				}
			\end{scope}
		\end{tikzpicture}
		\caption{An illustration of the triangle $\Delta_k$, including $k^2-1$ parallel lines with normal $\lambda$ that intersect the bottom line equidistantly. All other lengths are uniquely determined as a consequence of the intercept theorem.}\label{fig:triangle0}
	\end{figure}		
	
	We now set
	\begin{align*}
		\bar{u}_k(x)=
		\begin{cases}
			A_kx + \frac{j}{k}e_2 &\text{ if } x\in M_k^j \text{ for some } j\in\{1,\ldots,k^2\},\\
			u_{\lambda,e_2}(x) &\text { if } x\notin \Delta_k,
		\end{cases}
	\end{align*}
	for $x\in Q_{e_2}$. 
	The task is now to carefully estimate the jumps of $u_k$ on $\partial \Delta_k$ as well as those within $\Delta_k$.
	Along the bottom edge $D_k^1$ of $\Delta_k$ in direction $e_1$, we compute with \eqref{Ak} that
	\begin{align}\label{D1}
		\int_{D_k^1} \! f([\bar{u}_k],\nu_{\bar{u}_k}) \dd \Hcal^{n-1} &= \sum_{j=1}^{k^2} \int_{\frac{j-1}{k^3}}^{\frac{j}{k^3}} f\big((-k^2 x_1 +\tfrac{j}{k})e_2,e_2\big)\dd x_1 =\sum_{j=1}^{k^2} \int_{\frac{j-1}{k^3}}^{\frac{j}{k^3}} \big(-k^2 x_1 +\tfrac{j}{k}\big)\,\dd x_1f(e_2,e_2) \nonumber\\
		&=  \sum_{j=1}^{k^2} \frac{1}{2k^4} f(e_2,e_2) = \frac{1}{2k^2}f(e_2,e_2) \leq \frac{1}{k^2} \max_{\Scal^{n-1}\times\Scal^{n-1}} f;
	\end{align}
	note that $-k^2 x_1 +\tfrac{j}{k}> 0$ on $(\frac{j-1}{k^3},\frac{j}{k^3})$, which allows us to exploit the positive $1$-homogeneity of $f$ in the first variable.
	With similar arguments, we estimate, with the help of \eqref{Ak} and \eqref{Akw}, along the top edge $D_k^2$ of $\Delta_k$, in direction $\bar{\xi}_k \coloneqq \frac{1}{|\xi_k|}\xi_k$ (and outer normal $\zeta_k$)
	\begin{align}\label{D2}
		\int_{D_k^2} f([\bar{u}_k],\nu_{\bar{u}_k}) \dd \Hcal^{n-1} &= \sum_{j=1}^{k^2} \int_{\frac{(j-1)|\xi_k|}{k^2}}^{\frac{j|\xi_k|}{k^2}} f\big(\lambda - k^2tA_k\bar{\xi}_k -\tfrac{j}{k}e_2,\zeta_k\big)\dd t \nonumber\\
		&= \sum_{j=1}^{k^2} \int_{\frac{(j-1)|\xi_k|}{k^2}}^{\frac{j|\xi_k|}{k^2}} f\big(\lambda + \tfrac{k}{|\xi_k|}te_2 - \tfrac{2}{|\xi_k|}t\lambda  -\tfrac{j}{k}e_2,\zeta_k\big)\dd t \nonumber\\
		&\leq \sum_{j=1}^{k^2} \int_{\frac{(j-1)|\xi_k|}{k^2}}^{\frac{j|\xi_k|}{k^2}} \Big[\big(1-\tfrac{2}{|\xi_k|}t\big)f(\lambda,\zeta_k) - \big(\tfrac{k}{|\xi_k|}t-\tfrac{j}{k}\big)f(-e_2,\zeta_k)\Big]\dd t \nonumber\\
		&= \frac{|\xi_k|}{2k}f(-e_2,\zeta_k) \leq \frac{|\xi_k|}{2k}\max_{\Scal^{n-1}\times\Scal^{n-1}} f \leq \frac{1}{k^2}\max_{\Scal^{n-1}\times\Scal^{n-1}} f.
	\end{align}
	On the shortest edge $D_k^3$, in direction $w\in\Scal^{n-1}$, we observe the total jump energy
	\begin{align}\label{D3}
		\int_{D_k^3} f([\bar{u}_k],\nu_{\bar{u}_k}) \dd \Hcal^{n-1} &= \int_0^\frac{2}{k^2} f\big(\lambda + k e_2 - k^2t\lambda - ke_2,\lambda\big)\dd t \nonumber\\
		&= f(\lambda,\lambda)\int_0^\frac{1}{k^2} (1-k^2t)\,\dd t + f(-\lambda,\lambda)\int_{\frac{1}{k^2}}^{\frac{2}{k^2}} (k^2t-1)\,\dd t \nonumber\\
		&= \frac{1}{2k^2}\left(f(\lambda,\lambda) + f(-\lambda,\lambda)\right) \leq \frac{1}{k^2}\max_{\Scal^{n-1}\times\Scal^{n-1}} f,
	\end{align}
	due to \eqref{Akw}.
	The parallel lines inside the triangle $\Delta_k$, which we denote by $L_k^j$ and have length $\frac{2j}{k^4}$ for $j=1,\ldots,k^2-1$, yield the jumps
	\begin{align}\label{inside}
		\sum_{j=1}^{k^2-1} \int_{L_k^j} f([\bar{u}_k],\nu_{\bar{u}_k}) \dd \Hcal^{n-1} &= \sum_{j=1}^{k^2-1} \frac{j}{k^4}f\big(\tfrac{1}{k}e_2,\lambda)\nonumber \\
		&= \frac{2}{k^5}\sum_{j=1}^{k^2-1} j f(e_2,\lambda) = \frac{k^4 - k^2}{k^5}f(e_2,\lambda) \leq \frac{1}{k}f(e_2,\lambda).
	\end{align}
	\smallskip
	
	If $w \cdot e_2 <0$, then we redefine $\Delta_k$ as having the vertices $0$, $-\frac{1}{k}e_1$, and $-\frac{1}{k}e_1 - \frac{2}{k^2}w$; the parallel lines $L_k^j$ for $j=1,\ldots,k^2-1$ inside $\Delta_k$ are drawn analogously. The function $\bar{u}_k$ is now set to be
	\begin{align*}
		\bar{u}_k(x)=
		\begin{cases}
			-A_kx - \frac{j}{k}e_2 &\text{ if } x\in M_k^j \text{ for some } j\in\{1,\ldots,k^2\},\\
			u_{\lambda,e_2}(x) &\text { if } x\notin \Delta_k,
		\end{cases}
	\end{align*}
	for $x\in Q_{e_2}$. Using the evenness of $f$, it is evident that the previous estimates still hold.
	
	\medskip
	
	\textit{Step 2: Extending the construction to multiple triangles ($n=2$).}
	We now choose $k=2N$ for $N\in \N$ and place $k-2 = 2N-2$ copies of $\Delta_k$ next to each other, i.e., we set $\Delta_k^i \coloneqq \Delta_k + \frac{i}{2N}e_1$ for $i\in\{1-N,\ldots,N-2\}$, cf.~Figure \ref{fig:copies}.
	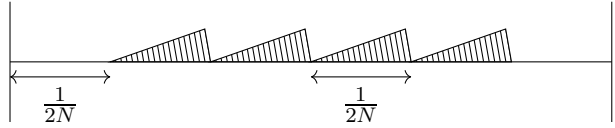
\begin{figure}[h!]
		\centering
		\begin{tikzpicture}[scale=8]
			\def\k{3}		
			\draw (-1/2,0) -- (1/2,0);
			\draw (-1/2,-.1) --++ (0,.2);
			\draw (1/2,-.1) --++ (0,.2);
			\draw[<->] (-1/2,-.025) --++ ({1/(4*\k)},0) node[below] {$\frac{1}{2N}$} --++ ({1/(4*\k)},0);
			\draw[<->] (0,-.025) --++ ({1/(4*\k)},0) node[below] {$\frac{1}{2N}$} --++ ({1/(4*\k)},0);
			\foreach \i in {-2,...,1}{
			\draw ({\i/(2*\k)},0) -- ({(\i+1)/(2*\k)},0) -- ($({(\i+1)/(2*\k)},0) + (100:{1/(2*\k*\k)})$)  -- cycle;
			\begin{scope}
				\clip ({\i/(2*\k)},0) -- ({(\i+1)/(2*\k)},0) -- ($({(\i+1)/(2*\k)},0) + (100:{1/(2*\k*\k)})$) -- cycle;
				\foreach \j in {0,...,36}{
					\draw ($({\i/(2*\k)},0)+ ({\j/(4*\k*\k*\k)},0)$) --++ (100:{1/(2*\k*\k)});
				}
			\end{scope}	
			}
		\end{tikzpicture}
		\caption{An illustration of the translated copies $\Delta_k^i$ of $\Delta_k$.}\label{fig:copies}
	\end{figure}
	Moreover, we define
	\begin{align*}
		u_k(x) =
		\begin{cases}
			\bar{u}_k(x-\frac{i}{2N}e_1) & \text{ if } x\in \Delta_k^i \text{ for some }i\in\{1-N,\ldots,N-2\},\\
			u_{\lambda,e_2}(x) &\text{ otherwise,}
		\end{cases}
	\end{align*}
	for $x\in Q_{e_2}$.
	In light of \eqref{D1} - \eqref{inside} and Figure \ref{fig:copies}, we then obtain
	\begin{align*}
		\int_{J_{u_k}} f\big([u_k],&\nu_{u_k}\big) \dd \Hcal^{n-1} = \\
		&=\frac{2}{k}f(\lambda,e_2)  + (k-2)\left(\int_{\partial \Delta_k} f([\bar{u}_k],\nu_{\bar{u}_k}) \dd \Hcal^{n-1} + \sum_{j=1}^{k^2-1} \int_{L_k^j} f([\bar{u}_k],\nu_{\bar{u}_k}) \dd \Hcal^{n-1}\right)\\
		&\leq \frac{2}{k}f(\lambda,e_2) + (k-2)\left(\frac{3}{k^2}\max_{\Scal^{n-1}\times\Scal^{n-1}} f + \frac{1}{k}f(e_2,\lambda) \right).
	\end{align*}
	For $k\to \infty$, this estimate and the ${\rm BD}$-ellipticity of $f$ yield
	$$f(\lambda,e_2) \leq \limsup_{k\to\infty}\int_{J_{u_k}} f\big([u_k],\nu_{u_k}\big) \dd \Hcal^{n-1} \leq f(e_2,\lambda).$$	
	
	\medskip
	
	\textit{Step 3: Generalization to higher dimensions ($n\geq 3$).}
	Let $n\in\N$ with $n\geq 3$. We aim to prove again that $f(\lambda,e_2) \leq f(e_2,\lambda)$ for any $\lambda\in \Scal^{n-1}$.
	To this end, let us consider the two-dimensional plane $H$ spanned by $\lambda$ and $e_2$, choose a vector $v\in \Scal^{n-1}\cap N_{e_2}\cap H$ with $N_{e_2}=\{x\in Q_{e_2}: x\cdot e_2 = 0\}$, and select the unique matrix $A_k\in \R^{n\times n}_{\skw}$ with
	\begin{align}\label{Ak2}
		A_kv = -k^2e_2\quad\text{and}\quad A_k z=0\text{ for every } z\in H^\perp,
	\end{align}
	where $H^\perp$ is the orthogonal complement of $H$. This matrix describes a rotation by $\frac{-\pi}{2}$ and scaling by $k^2$ in the plane $H$.
	If $\lambda$ lies in the $e_1$-$e_2$-plane then $A_k$ can be chosen as in the Steps 1 and 2 and filled up with zeroes in the remaining components.
	Any other case is handled by the fact that $U^TAU\in \R^{n\times n}_{\skw}$ for every $A\in\R^{n\times n}_{\skw}$ and every orthogonal matrix $U\in O(2)$.
	As before, we select $w\in\Scal^{n-1}$ such that \eqref{Akw} is satisfied and assume that $w\cdot e_2>0$; the case $w\cdot e_2<0$ can be addressed analogously.
	
	In the $\lambda$-$e_2$-plane, we can set up a triangle $\Delta_k$ as in Step 1, where $e_1$ is replaced by $v$.
	We introduce parallel right prisms $\Delta_{k,n}^i$, each base of the shape $\Delta_k$, where $i$ is taken from an index set $I$ of consecutive integers with cardinality $k$. 
	These prisms are constructed in such a way that $\dist(\Delta_{k,n}^i,\partial Q_{e_2})>\frac{1}{k}$ and their union covers the hyperplane $N_{e_2}$ up to an error in $\Hcal^{n-1}$-measure of order $\frac{1}{k}$.
	Moreover, we intersect each prism by $k^2-1$ hyperplanes with normal $\lambda$; the intersection of these planes with the prism $\Delta_{k,n}^i$ are called $L_{k,n}^{j,i}$ and we denote the resulting geometric subfigures by $M_{k,n}^{j,i}$ for $j\in\{1,\ldots,k^2\}$. 
	We then define
	\begin{align*}
		u_k(x)=
		\begin{cases}
			A_k(x-\tfrac{i}{k}v) + \frac{j}{k}e_2 &\text{ if } x\in M_{k,n}^{j,i} \text{ for some } j\in\{1,\ldots,k^2\},\, i\in I,\\
			u_{\lambda,e_2}(x) &\text { otherwise},
		\end{cases}
	\end{align*}
	for $x\in Q_{e_2}$.
	
	The calculations for the occurring jumps are now simple modifications of those in Steps 1 and 2 due to \eqref{Akw} and \eqref{Ak2}.
	For instance, in the analogue case of \eqref{inside}, we can write $L_{k,n}^{j,i} = (L_k^j+\frac{i}{k}v)\times C_{k,n}^i$ with $(n-2)$-dimensional cuboids $C_{k,n}^i\subset H^\perp$. 
	The jump energy across these surfaces inside the prism $\Delta_{k,n}^{i}$ is given by
	\begin{align*}
		\sum_{j=1}^{k^2-1}\int_{L_{k,n}^{j,i}}f\big([u_k],\nu_{u_k}\big) \dd \Hcal^{n-1} = \sum_{j=1}^{k^2-1}\int_{L_k^j\times C_{k,n}^{i}} f\big(\tfrac{1}{k}e_2,\lambda\big) \dd \Hcal^{n-1}\leq \Hcal^{n-2}(C_{k,n}^{i})\frac{1}{k}f(e_2,\lambda).
	\end{align*}
	After taking the sum over all $i\in I$, the right-hand side converges to $f(e_2,\lambda)$, since the expression $\sum_{i\in I}\Hcal^{n-2}(C_{k,n}^{i})\frac{1}{k}$ converges to the $\Hcal^{n-1}$-measure of $N_{e_2}$ up to an error of $\frac{1}{k}$. The computations for the remaining surfaces, that is, the analogues of \eqref{D1} - \eqref{D3}, are even easier since the cuboids $C_{k,n}^{i}$ are bounded and the cardinality of $I$ is of order $k$.
	
	Additionally, one needs to account for the jumps across the two bases $\Delta_k^i$ and $\bar{\Delta}_k^i$ of each prism, which are described by the triangle $\Delta_k$.
	If $n=3$ and $z\in H^\perp$ with $|z|=1$, then the energy contribution of the jumps across these two triangles with normal $z$ (or $-z$) vanishes in the limit. 
	Indeed, since the lines $L_k^j$ from Step 1 have length $\frac{2j}{k^4}$ for all $j\in\{1,\ldots,k^2-1\}$ and the edge of $\Delta_k$ with normal $\lambda$ has length $\frac{2}{k^2}$, we find that
	\begin{align*}
		\int_{\Delta_k^i} f\big([u_k],\nu_{u_k}\big) \dd \Hcal^{n-1} &=
		\sum_{j=1}^{k^2} \int_{\frac{j-1}{k^3}}^{\frac{j}{k^3}} \int_0^{\frac{2j}{k^4}} f\big(\lambda - A_k(t v + sw) - \tfrac{j}{k}e_2,z\big) \dd s \dd t\\
		&=\sum_{j=1}^{k^2} \int_{\frac{j-1}{k^3}}^{\frac{j}{k^3}} \int_0^{\frac{2j}{k^4}} f\big(\lambda + k^2t e_2 - k^2s\lambda - \tfrac{j}{k}e_2,z\big) \dd s \dd t\\
		&\leq \sum_{j=1}^{k^2} \int_{\frac{j-1}{k^3}}^{\frac{j}{k^3}} \int_0^{\frac{2}{k^2}} C(1 + k^2t + k^2 s + \tfrac{j}{k})\dd s \dd t\\
		&\leq \sum_{j=1}^{k^2} \int_{\frac{j-1}{k^3}}^{\frac{j}{k^3}} \int_0^{\frac{2}{k^2}} C(3 +2k)\dd s \dd t
		= \frac{2C(3+2k)}{k^3},
	\end{align*}
	while exploiting \eqref{linear_bound}, as well as \eqref{Akw} and \eqref{Ak2}; analogously for $\bar{\Delta}_k^i$.
	The energy contribution at the combined prism bases therefore vanishes in the limit, considering that the cardinality of the index set $I$ is of order $k$.
	For $n\geq 4$, the bases of the prisms are sets of zero $\Hcal^{n-1}$-measure and thus, can be neglected when calculating the surface energy.
\end{proof}

Next, we prove the ${\rm BD}$-elliptic generalization of \cite[Lemma 6.2]{Sil17}, see also \eqref{AmB90_inequality}, by tailoring \v{S}ilhav\'y's construction to our setting.

\begin{lemma}\label{lem:silhavy_construction}
	Let $f\colon \R^n\times \Scal^{n-1}\to [0,\infty)$ be even, positively $1$-homogeneous in the first variable and ${\rm BD}$-elliptic.
	Then, for any $(\lambda,\eta)\in\R^n\times\Scal^{n-1}$ it holds that
	$$f(\lambda,\eta) \leq \sum_{i=1}^mf(\lambda_i,\eta_i)$$
	for all $(\lambda_i,\eta_i)\in\R^n\times \Scal^{n-1}$ for $i\in\{1,\ldots,m\}$ and $m\in\N$ such that $\sum_{i=1}^{m}\lambda_i\otimes \eta_i = \lambda\odot\eta$.
\end{lemma}
\begin{proof}
	Let $(\lambda,\eta)\in\R\times \Scal^{n-1}$ with
	\begin{align}\label{decomposition}
		\lambda\odot\eta =\sum_{i=1}^m\lambda_i\otimes\eta_i,\quad\text{or equivalently}\quad \lambda\otimes \eta = 2\sum_{i=1}^m \lambda_i\otimes\eta_i - \eta\otimes \lambda,
	\end{align}
	for a collection $(\lambda_i,\eta_i)\in\R^{n}\times\Scal^{n-1}$ for $i=1,\ldots,m$.
	We consider for $k\in\N$ with $k> 2$ the rectangle
	$$B_k \coloneqq \left\{x\in \R^{n}: 0\leq x\cdot \eta \leq \frac{1}{k}, -\frac{1}{2}+\frac{1}{2k}\leq  x\cdot \eta^\perp \leq \frac{1}{2}-\frac{1}{2k}\right\}$$
	and define
	$$u_k(x) \coloneqq \begin{cases}
		v_k(x) &\text{ if } x\in B_k,\\
		u_{\lambda,\eta}(x)	&\text{ if } x\notin B_k,
	\end{cases}\quad\text{with}\quad v_k(x)  \coloneqq  \sum_{i=1}^{m}\frac{1}{k}\lambda_i\langle k^2x\cdot \eta_i\rangle -k(\eta\otimes\lambda)^{\skw}x$$
	for $x\in Q_\eta$; here, $(\eta\otimes\lambda)^{\skw} \coloneqq \frac12\big(\eta\otimes\lambda-\lambda\otimes\eta\big)$, and the notation $\langle r\rangle$ stands for the integer part of $r\in\R$. In particular, it holds that, 	for any $t\in\R$ and $n\in\N$,
	\begin{align}\label{integral_part}
		0\leq t - \frac{1}{n}\langle nt\rangle \leq \frac{1}{n}.
	\end{align}
	Note that \eqref{decomposition} and \eqref{integral_part} then yield that
	\begin{align}
		|k\lambda(x\cdot\eta) - v_k(x)| &= \left|k(\lambda\otimes\eta)x - \sum_{i=1}^{m}\frac{1}{k}\lambda_i\langle k^2x\cdot \eta_i\rangle +k(\eta\otimes\lambda)^{\skw}x\right|\nonumber\\
		&=k\left|2\sum_{i=1}^m\lambda_i(x\cdot\eta_i) - (\eta\otimes\lambda)x - \sum_{i=1}^{m}\frac{1}{k^2}\lambda_i\langle k^2x\cdot \eta_i\rangle +(\eta\otimes\lambda)^{\skw}x\right|\nonumber\\
		&\leq \frac{1}{k}\sum_{i=1}^m\left|\lambda_i\right| + \left|\sum_{i=1}^m\lambda_i(x\cdot\eta_i) - (\eta\otimes\lambda)x +(\eta\otimes\lambda)^{\skw}x\right|\nonumber\\
		&= \frac{1}{k}\sum_{i=1}^m\left|\lambda_i\right| + \left|(\lambda\odot\eta)x - (\eta\otimes\lambda)x +(\eta\otimes\lambda)^{\skw}x\right| = \frac{1}{k}\sum_{i=1}^m\left|\lambda_i\right|\label{difference_estimate}
	\end{align}		
	for every $x\in Q_\eta$.
	
	The jump set $J_{u_k}$ of $u_k$ can be written as the union
	$$J_{u_k}=L_k\cup M_k \cup N_k \cup R_k \cup S_k$$
	with
	\begin{align*}
		L_k &= \bigcup_{i=1}^m L_k^i\quad\text{with}\quad L_k^i=\left\{x\in B_k: k^2x\cdot\eta_i\in\Zb\right\},\\
		M_k &= \left\{x\in \partial B_k: 0<x\cdot\eta < \frac{1}{k}\right\},\\
		N_k &= \left\{x\in Q_\eta: x\cdot\eta = 0,\, \frac{1}{2}-\frac{1}{2k}<|x\cdot \eta^\perp| < \frac{1}{2}\right\},\\
		R_k &= \left\{x\in\partial B_k: x\cdot\eta = 0\right\},\\
		S_k &= \left\{x\in\partial B_k: x\cdot\eta = \frac{1}{k}\right\}.
	\end{align*}
	We now compute the energy of the jumps at these interfaces.
	Just as in the proof of \cite[Lemma 6.2]{Sil17}, we obtain on $L_k$ that
	\begin{align}\label{jump_L_k}
		\int_{L_k}f\big([u_k],\nu_{u_k}\big) \dd\Hcal^{n-1} \leq \frac{1}{k}\sum_{i=1}^mf(\lambda_i,\eta_i)\Hcal^{n-1}(L_k^i)\to \sum_{i=1}^m f(\lambda_i,\eta_i)\quad\text{as }k\to\infty,
	\end{align}
	where we have used that $f$ is subadditive in light of Proposition \ref{prop:subadditivity}.
	On $N_K$, we compute
	\begin{align}\label{jump_N_k}
		\int_{N_k}f\big([u_k],\nu_{u_k}\big) \dd\Hcal^{n-1} = \frac{1}{k}f(\lambda,\eta)\to 0\quad\text{as }k\to\infty,
	\end{align}
	and for $x\in M_k$, we exploit \eqref{difference_estimate} to estimate
	\begin{align*}
		|\lambda-v_k(x)| &\leq |\lambda - k\lambda(x\cdot\eta)| + |k\lambda(x\cdot\eta) - v_k(x)|\\
		&\leq (1+k|(x\cdot\eta)|)|\lambda| + \frac{1}{k}\sum_{i=1}^m\left|\lambda_i\right| \leq 2|\lambda| + \frac{1}{k}\sum_{i=1}^m\left|\lambda_i\right| \leq 2|\lambda|.
	\end{align*}
	Hence, $|\lambda-v_k(x)|$ is bounded uniformly in $k$ and
	\begin{align}\label{jump_M_k}
		\int_{M_k}f\big([u_k],\nu_{u_k}\big) \dd\Hcal^{n-1} \to 0\quad\text{as } k\to\infty
	\end{align}
	since $\Hcal^{n-1}(M_k)$ vanishes in the limit.
	If $x\in R_k$, we use \eqref{difference_estimate} to obtain
	\begin{align*}
		|v_k(x)| = |v_k(x) - k\lambda(x\cdot\eta)| \leq \frac{1}{k}\sum_{i=1}^m|\lambda_i|,
	\end{align*}		
	and if $x\in S_k$, we similarly find
	\begin{align*}
		|\lambda- v_k(x)| = |k\lambda(x\cdot\eta) - v_k(x)| \leq \frac{1}{k}\sum_{i=1}^m|\lambda_i|,
	\end{align*}		
	which results in
	\begin{align}\label{jump_SR_k}
		\int_{R_k\cup S_k}f\big([u_k],\nu_{u_k}\big) \dd\Hcal^{n-1} \to 0 \quad\text{as } k\to\infty.
	\end{align}
	Combining \eqref{jump_L_k} - \eqref{jump_SR_k}, we then conclude the desired inequality since $f$ is ${\rm BD}$-elliptic and
	$u_k\in u_{\lambda,\eta} + \PR$.
\end{proof}

We are now in a position to prove our second main result of this paper.

\begin{theorem}[Characterization of $\rm \bm{BD}$-ellipticity]\label{theo:BD_symmetric_biconvex}
	Let $f\colon \R^n\times \Scal^{n-1}\to [0,\infty)$ be even and positively $1$-homogeneous in the first variable. Then, $f$ is ${\rm BD}$-elliptic if and only if $f$ is symmetric biconvex with 
	$$f(\lambda,\eta)=\Phi_f(\lambda\odot\eta)$$ 
	for every $(\lambda,\eta)\in\R^n\times\Scal^{n-1}$, where $\Phi_f$ is given as in \eqref{Phi1}.
\end{theorem}
\begin{proof}
	\textit{Step 1: ${\rm BD}$-ellipticity implies symmetric biconvexity.}
	If $f$ is ${\rm BD}$-elliptic, then it holds that $\bar{f}(\lambda,\eta)=\bar{f}(\eta,\lambda)$ (cf.~\eqref{1hom_extension}) for every $(\lambda,\eta)\in\R^n\times\Scal^{n-1}$ due to Lemma \ref{prop:switch}.
	In particular, we find that
	\begin{align}\label{f_Phi_estimate}
		\Phi_f(\lambda\odot\eta) \leq \bar{f}(\tfrac{1}{2}\lambda,\eta) + \bar{f}(\tfrac{1}{2}\eta,\lambda) = f(\lambda,\eta)
	\end{align}
	for every $(\lambda,\eta)\in\R^n\times\Scal^{n-1}$ since $\bar{f}$ is positively $1$-homogeneous in both variables.
	In view of Lemma \ref{lem:silhavy_construction}, we conclude that the two sides of \eqref{f_Phi_estimate} coincide, which proves that
	$f$ is symmetric biconvex.
	\medskip
	
	\textit{Step 2: Symmetric biconvexity implies ${\rm BD}$-ellipticity.}
	The proof is essentially a reformulation and simplification of some results in \cite{FPS21}. 
	For the reader's convenience, we present the details below.
	If $f$ is symmetric biconvex, then there exists a positively $1$-homogeneous, convex function $\Psi\colon \R^{n\times n}_{\sym}\to[0,\infty)$
	such that $f(\lambda,\eta)=\Psi(\lambda\odot\eta)$ for every $(\lambda,\eta)\in\R^n\times\Scal^{n-1}$. 
	
	Following the proof of \cite[Proposition 4.9]{FPS21}, which is based on \cite[Proposition 2.31]{AFP00}, we find a sequence of symmetric matrices $(A_i)_i\subset\R^{n\times n}_{\sym}$ such that $\Psi(F)=\sup_{i\in\N} A_i :F$ for every $F\in\R^{n\times n}_{\sym}.$
	In particular, it holds that
	\begin{align}\label{f_A_i}
		f(\lambda,\eta) = \sup_{i\in\N} (A_i\lambda)\cdot\eta\quad\text{for every } (\lambda,\eta)\in\R^n\times\Scal^{n-1}.
	\end{align}
	Finally, we define the auxiliary functions $g_i(x) \coloneqq A_ix$ for every $x\in\R^{n}$ and continue with the strategy in \cite[Theorem 3.4]{FPS21}.
	We fix $i\in\N$, $(\lambda,\eta)\in\R^n\times\Scal^{n-1}$ with $\lambda\neq 0$ (otherwise there is nothing to show), and select any $u\in u_{\lambda,\eta} + \PR$. Since $u\in {\rm SBV}(Q_\eta;\R^n)$ and $g_i\in C^1(\R^n;\R^n)$, we may apply the chain rule \cite[Theorem 3.96]{AFP00} to differentiate the composition $g_i\circ u\in {\rm BV}(Q_\eta;\R^n)$, obtaining
	\begin{align*}
		D(g_i\circ u) &= \nabla g_i(u)\nabla u\Lcal^n + \big(g_i(u^+) - g_i(u^-)\big)\otimes\nu_u\Hcal^{n-1}\mres J_u\\
			&=A_i\nabla u\Lcal^n + (A_i[u])\otimes\nu_u\Hcal^{n-1}\mres J_u\,,
	\end{align*}
	where $D(g_i\circ u)$ is the distributional derivative. 
	By evaluating in $Q_\eta$ and taking the trace, we then find
	\begin{align*}
		\Tr\big(D(g_i\circ u)(Q_\eta)\big) = \int_{Q_\eta}A_i : (\nabla u)^T\dd \Lcal^{n} + \int_{J_u}(A_i[u])\cdot \nu_u\dd \Hcal^{n-1}.
	\end{align*}
	Since $A_i$ is symmetric and $(\nabla u)^T$ is skew-symmetric, their scalar product vanishes.
	Moreover, as $g_i\circ u - g_i \circ u_{\lambda,\eta}$ has compact support in $Q_\eta$ it holds that $D(g_i\circ u)(Q_\eta) = D(g_i\circ u_{\lambda,\eta})(Q_\eta)$, which leads to
	\begin{align*}
		\int_{J_u}(A_i[u])\cdot\nu_u \dd \Hcal^{n-1} = (A_i\lambda)\cdot\eta.
	\end{align*}
	Finally, we conclude from \eqref{f_A_i} that
	\begin{align*}
		\int_{J_u}f([u],\nu_u)\dd \Hcal^{n-1} &= \int_{J_u}\sup_{i\in\N}(A_i[u])\cdot\eta \dd \Hcal^{n-1}\\
		&\geq \sup_{i\in\N}\int_{J_u}(A_i[u])\cdot\eta \dd \Hcal^{n-1} =\sup_{i\in\N} (A_i\lambda)\cdot\eta = f(\lambda,\eta),
	\end{align*}	 
	which proves that $f$ is ${\rm BD}$-elliptic.
\end{proof}

Now that we have established the equivalence of ${\rm BD}$-ellipticity and symmetric biconvexity under suitable assumptions, we turn to a relaxation result similar to Proposition \ref{prop:BV_envelope}.
For an even function $f\colon \R^n\times\Scal^{n-1}\to[0,\infty)$ we define the two ${\rm BD}$-elliptic envelopes
\begin{align*}
	f^\BD(\lambda,\eta)& \coloneqq \sup\left\{h(\lambda,\eta) : h\text{ is ${\rm BD}$-elliptic and } h\leq f\right\},\\
	f_\BD(\lambda,\eta)& \coloneqq \inf\left\{\int_{J_u} f([u],\nu_u)\dd\Hcal^{n-1}: u \in u_{\lambda,\eta} + \PR\right\}.
\end{align*}
for every $(\lambda,\eta)\in\R^n\times\Scal^{n-1}$.

\begin{proposition}[$\rm \bm{BD}$-elliptic envelope]\label{prop:BD_envelope}
	Let $f\colon \R^n\times \Scal^{n-1}\to [0,\infty)$ be even and positively $1$-homogeneous in the first variable.
	Then, it holds that
	$$f^\BD(\lambda,\eta) = f_\BD(\lambda,\eta) = \Phi_f(\lambda\odot\eta),$$
	for every $(\lambda,\eta)\in\R^n\times\Scal^{n-1}$.
	In particular, the ${\rm BD}$-elliptic envelope of $f$ is ${\rm BD}$-elliptic, and $f$ is ${\rm BD}$-elliptic if and only if it coincides with its ${\rm BD}$-elliptic envelope.
\end{proposition}
\begin{proof}
	The proof can be handled analogously to Proposition \ref{prop:BV_envelope}.
	We merely replace Theorem \ref{theo:BV_elliptic} with Theorem \ref{theo:BD_symmetric_biconvex}, and \cite[Lemma 6.2]{Sil17} with Lemma \ref{lem:silhavy_construction}.
\end{proof}

Similar to Remark \ref{rem:joint_convexity}, we tackle one final related convexity notion.
\begin{remark}[Symmetric joint convexity]
	We say that an even function $f	\colon \R^n\times\Scal^{n-1}\to[0,\infty)$ is symmetric jointly convex if
	\begin{align*}
		f(\lambda,\eta) = \sup_{i\in\N}\big(g_i(\lambda) - g_i(0)\big)\cdot \eta\quad\text{ for every } (\lambda,\eta)\in\R^n\times\Scal^{n-1},
	\end{align*}
	where $g_i\in\C^1( \R^n;\R^n)$ is Lipschitz continuous and conservative for every $n\in\N$.
	
	Here, we merged \cite[Definition 3.1]{FPS21} with the class of functions for $g_i$ used in \cite[Remark 3.2]{FPS21}.
	In our setting the chain rule for compositions of $g_i$ with ${\rm BV}$-functions can be applied directly.
	Moreover, if $f$ is positively $1$-homogeneous in the first variable, then symmetric joint convexity is also equivalent to ${\rm BD}$-ellipticity.

	The proofs of both implications can be handled almost exactly as in Step 2 of the proof of Theorem \ref{theo:BD_symmetric_biconvex}, which are inspired by \cite[Theorem 3.4]{FPS21} and \cite[Proposition 4.9]{FPS21} but do not require boundedness of the functions $g_i$.	
\end{remark}

We close this article with a curious example for a symmetric biconvex function. It has already been establishes in \cite[Example 4.16]{FPS21} that densities of the form
$$f(\lambda,\eta) = \psi(\lambda),\quad (\lambda,\eta)\in\R^{n}\times\Scal^{n-1}$$
with an anisotropic function $\psi$ are, in general, not ${\rm BD}$-elliptic. In the following, we tackle the case $\psi=|\cdot|$, which has been addressed in \cite[Theorem 4.1]{FPS21}
\begin{example}\label{ex:BD}
	We consider the function
	$$f\colon \R^{2}\times\Scal^{1}\to[0,\infty),\ (\lambda,\eta)\mapsto |\lambda\otimes\eta|,$$
	where $|.|$ is the Frobenius norm.
	It is then obvious that $f$ is biconvex, and also ${\rm BV}$-elliptic due to Theorem \ref{theo:BV_elliptic}, since $f(\lambda,\eta)=\Phi(\lambda\otimes\eta)$ for every $(\lambda,\eta)\in\R^2\times\Scal^{1}$ with
	$$\Phi \colon \R^{2\times 2}\to [0,\infty),\ F\mapsto |F|.$$
	It might be surprising to see that $f$ is also symmetric biconvex, and thus ${\rm BD}$-elliptic in view of Theorem \ref{theo:BD_symmetric_biconvex}, since $f(\lambda,\eta)=\Psi(\lambda\odot\eta)$ for every $(\lambda,\eta)\in\R^2\times\Scal^{1}$ with
	$$\Psi \colon \R^{2\times 2}_{\sym}\to [0,\infty),\ F\mapsto \sqrt{(F_{11} - F_{22})^2 + (F_{12}+F_{21})^2} = \sqrt{|F|^2 - 2\det F}.$$
	However, neither $\Phi$ nor $\Psi$ (extended canonically to all of $\R^{2\times 2}$) coincides with $\Phi_f$ (cf.~\eqref{Phi1}) on all of $\R^{2\times 2}$.
	
	It turns out that $\Phi_f$ is the nuclear norm $|\cdot|_\ast$ on $\R^{2\times 2}$, i.e.,
	$$\Phi_f(F)= |F|_\ast \coloneqq  \Tr\big(\sqrt{F^TF}\big) = \sigma_1(F) + \sigma_2(F)$$
	for every $F\in\R^{2\times 2}$, where $\sigma_1(F),\sigma_2(F)\geq 0$ are the two singular values of $F$.
	To prove this identity, we turn our attention to Lemma \ref{lem:Phi_f}, in which we established that $\Phi_f$ is the convex envelope of \eqref{phi_extension}. 
	Since the nuclear norm is convex and coincides with $f$ on tensor products, we obtain the trivial inequality $|\cdot|_\ast \leq \Phi_f$.
	We establish the reverse inequality by exploiting the singular value decomposition: for every $F\in\R^{2\times 2}$ there exist orthogonal matrices $U,V\in {\rm O}(2)$ such that
	$$F=U\begin{pmatrix}\sigma_1(F) & 0 \\ 0 & \sigma_2(F)\end{pmatrix}V^T;$$
	in particular, it holds that
	$$F = \sigma(F)(Ue_1)\otimes(Ve_1) + \sigma_2(F)(Ue_2)\otimes(Ve_2).$$
	This composition then yields that $\Phi_f\leq |\cdot|_\ast$ since the columns of $U$ and $V$ are normalized and the singular values are non-negative.
\end{example}

\subsection*{Acknowledgments}
The authors appreciate the insightful discussions with Manuel Friedrich about ${\rm BD}$-ellipticty and symmetric joint convexity, and would like to thank Hidde Sch\"{o}nberger for his valuable suggestions for Example \ref{ex:BD}.
This work was initiated when CK and DE were affiliated with Utrecht University.
MM is a member of the \emph{Gruppo Nazionale per l'Analisi Matematica, la Probabilit\`{a} e le loro Applicazioni} (GNAMPA) of the Italian \emph{Istituto Nazionale di Alta Matematica} (INdAM) and of the \emph{Integrated Additive Manufacturing} research group at Politecnico di Torino. He acknowledges support from the MIUR grant Dipartimenti di Eccellenza 2018-2022 (E11G18000350001). 
This study was carried out within the \emph{Geometric-Analytic Methods for PDEs and Applications} project (2022SLTHCE), funded by European Union -- Next Generation EU within the PRIN 2022 program (D.D.~104 - 02/02/2022 Ministero dell’Università e della Ricerca). This manuscript reflects only the authors’ views and opinions and the Ministry cannot be considered responsible for them.

Part of this research was developed at CIRM Luminy, France, during the workshop \emph{Beyond Elasticity: Advances and Research Challenges}, and at ESI Vienna, Austria, during the workshop \emph{Between Regularity and Defects: Variational and Geometrical Methods in Materials Science}. The hospitality of both centers is gratefully acknowledged.

\bibliographystyle{abbrv}
\bibliography{SurfaceRelaxation}

\end{document}